\documentclass[leqno,11pt]{amsart}
\usepackage{amsmath,amssymb}
\usepackage{color}
\usepackage{comment}
\usepackage{todonotes}

\allowdisplaybreaks

\theoremstyle{plain}
\newtheorem{theorem}{Theorem}[section]
\newtheorem{lemma}[theorem]{Lemma}
\newtheorem{proposition}[theorem]{Proposition}
\newtheorem{corollary}[theorem]{Corollary}
\newtheorem*{theorem-A}{Theorem A}
\newtheorem*{theorem-B}{Theorem B}
\newtheorem*{theorem-C}{Theorem C}
\newtheorem*{theorem-D}{Theorem D}
\newtheorem*{theorem-E}{Theorem E}
\newtheorem*{theorem-F}{Theorem F}
\theoremstyle{definition}

\newtheorem{remark}[theorem]{Remark}

\newtheorem{example}[theorem]{Example}

\newtheorem{theoremalph}{Theorem}

\def\rn{\mathbb R\sp n}
\def\Rn{\mathbb R\sp n}
\def\R{\mathbb R}
\def\N{\mathbb N}

\def\M{\mathcal M}

\def\limsup{\operatornamewithlimits{lim\,sup}}

\def\r{\mathbb R}

\newcommand{\medint}{-\kern  -,435cm\int}
\newcommand{\medintinrigo}{-\kern  -,315cm\int}
\newcommand{\medelle}{-\kern  -,235cm L}
\newcommand{\medellenrigo}{-\kern  -,180cm L}

\def\EE{E}
\def\FF{F}

\newtoks\by
\newtoks\paper
\newtoks\book
\newtoks\jour
\newtoks\yr
\newtoks\pages
\newtoks\vol
\newtoks\publ
\newtoks\eds
\newtoks\proc
\newtoks\no
\def\ota{{\hbox{???}}}
\def\cLear{\by=\ota\paper=\ota\book=\ota\jour=\ota\yr=\ota
\pages=\ota\vol=\ota\publ=\ota}
\def\endpaper{\the\by, \textit{\the\paper},
{\the\jour} \textbf{\the\vol} (\the\yr), \the\pages.\cLear}
\def\endbook{\the\by, \textit{\the\book}, \the\publ.\cLear}
\def\endprep{\the\by, \textit{\the\paper}, \the\jour.\cLear}
\def\endproc{\the\by, \textit{\the\paper}, \the\publ, \the\pages.\cLear}
\def\name#1#2{#1 #2}
\def\et{ and }

\numberwithin{equation}{section}

\newcommand{\norm}[1]{{\left\vert\kern-0.25ex\left\vert\kern-0.25ex\left\vert #1
    \right\vert\kern-0.25ex\right\vert\kern-0.25ex\right\vert}}

\usepackage[normalem]{ulem}
\usepackage{soul}
\usepackage{cancel}
\newcommand\Del[1]{{\color{red}\ifmmode\cancel{#1}\else\sout{#1}\fi}}

\begin{document}

\title[Modulus of continuity]{On the modulus of continuity of fractional Orlicz-Sobolev functions}

\author {Angela Alberico, Andrea Cianchi, Lubo\v s Pick and Lenka Slav\'ikov\'a}

\address{Angela Alberico, Istituto per le Applicazioni del Calcolo ``M. Picone''\\
Consiglio Nazionale delle Ricerche \\
Via Pietro Castellino 111\\
80131 Napoli\\
Italy} \email{angela.alberico@cnr.it}

\address{Andrea Cianchi, Dipartimento di Matematica e Informatica \lq\lq U. Dini"\\
Universit\`a di Firenze\\
Viale Morgagni 67/a\\
50134 Firenze\\
Italy} \email{andrea.cianchi@unifi.it}

\address{Lubo\v s Pick, Department of Mathematical Analysis\\
Faculty of Mathematics and Physics\\
Charles University\\
Sokolovsk\'a~83\\
186~75 Praha~8\\
Czech Republic} \email{pick@karlin.mff.cuni.cz}

\address{Lenka Slav\'ikov\'a,
Department of Mathematical Analysis, Faculty of Mathematics and
Physics,  Charles University, Sokolovsk\'a~83,
186~75 Praha~8, Czech Republic}
\email{slavikova@karlin.mff.cuni.cz}
\urladdr{}


\begin{abstract}
Necessary and sufficient conditions are presented for a fractional Orlicz-Sobolev space on $\rn$ to be continuously embedded into a space of uniformly continuous functions. The optimal modulus of continuity is exhibited whenever these conditions are fulfilled. These results pertain to the supercritical Sobolev regime and complement earlier sharp embeddings into rearrangement-invariant spaces concerning the subcritical setting. Classical embeddings for fractional Sobolev spaces into H\"older spaces are recovered as special instances.
Proofs require novel strategies, since customary methods fail to produce optimal conclusions.
\end{abstract}

\maketitle

\section{Introduction}\label{intro}


The homogeneous fractional Orlicz-Sobolev spaces $V^{s,A}(\rn)$ are defined in terms of a smoothness parameter $s\in (0,\infty) \setminus \N$ and an integrability parameter Young function $A$. This class of   spaces extends the family of the standard homogeneous fractional Sobolev spaces, also called Gagliardo-Slobodeckij spaces, denoted by $V^{s,p}(\rn)$ and associated with Young functions of power type $t^p$, with $p \geq 1$.

Specifically, given $n\in \N$ and $s\in (0,1)$, the space $V^{s,A}(\rn)$ consists of all measurable functions $u$ in $\rn$ for which the Luxemburg type seminorm
\begin{equation}\label{intro0}
|u|_{s,A, \rn}
		= \inf\left\{\lambda>0: \int_{\rn} \int_{\rn}A\left(\frac{|u(x)-u(y)|}{\lambda|x-y|^s}\right)\frac{\,dx\,dy}{|x-y|^n}\le1\right\}\,
\end{equation}
is finite. Spaces of fractional order $s>1$ can be defined analogously, by replacing $u$ with its weak derivatives of order $[s]$, the integer part of $s$.

This definition was introduced in \cite{BonderSalort} and, as shown in earlier contributions of ours \cite{ACPS_lim0, ACPS_lim1, ACPS_emb, ACPS_inf}, allows for neat sharp Orlicz versions of classical results for the spaces  $V^{s,p}(\rn)$. Interestingly, setting $s$ as an integer in these versions exactly matches their counterparts for integer-order Orlicz-Sobolev spaces, although the latter are not recovered via the definition of $V^{s,p}(\rn)$ for integer $s$. This supports the aptness of the notion of fractional Orlicz-Sobolev spaces adopted here when compared to other definitions available in the literature.
Inhomogeneous spaces $W^{s,A}(\rn)$, whose members not only belong to $V^{s,A}(\rn)$, but also to the Orlicz space $L^A(\rn)$ together with their weak derivatives up to the order $[s]$, were studied in \cite{BreitCia}, where alternate equivalent norms are provided.

Information on the spaces $V^{s,A}(\rn)$ --  in particular on their embeddings --  is critical in view of their use in the analysis of nonlocal variational problems driven by Orlicz type nonlinearities. The study of problems of this kind and related properties of fractional Orlicz-Sobolev spaces has attracted the attention of various authors in recent years. In this connection, besides those mentioned above, see e.g. the contributions \cite{ACPS_Mazya, ACPS_camp, BaOu,  BKO, BKS,  CKW, FangZhang, BonderSalortVivas, KMS,  MSV, OSM, SalortVivas}.

Optimal embeddings of fractional Orlicz-Sobolev spaces into rearrangement-invariant Banach function spaces - membership in which only depends on global integrability properties of functions - are the subject of \cite{ACPS_emb} and \cite{ACPS_inf}. These papers deal with the subcritical and supercritical regimes, respectively. In particular, the latter pertains to couples $(s,A)$ for which the space $V^{s,A}(\rn)$ is continuously embedded into $L^\infty(\rn)$. It turns out that, for the same couples, every function from  $V^{s,A}(\rn)$ is also continuous.

In the present paper, we pursue a closer analysis of continuity properties of fractional Orlicz-Sobolev functions and provide quantitative information on their modulus of continuity. More precisely, necessary and sufficient conditions on $s \in (0,1)$ and $A$ are detected for a continuous embedding of the form
\begin{equation}\label{intro1}
V^{s,A}(\rn) \to C^{\omega (\cdot)}(\rn)
\end{equation}
to hold, where $C^{\omega (\cdot)}(\rn)$ denotes the space of functions in $\rn$ whose modulus of continuity does not exceed (a constant times) the function $\omega$. Moreover, the optimal modulus of continuity $\omega$ in embedding \eqref{intro1} is exhibited.
\\ When $s>1$, the same questions are discussed for the embedding
\begin{equation}\label{intro2}
V^{s,A}_{d,1}(\rn) \to C^{\omega (\cdot)}(\rn),
\end{equation}
where $V^{s,A}_{d,1}(\rn)$ denotes the subspace of $V^{s,A}(\rn)$ consisting of those functions all of whose weak derivatives up to the order $[s]$ decay near infinity in a weakest possible sense. The replacement of the space $V^{s,A}(\rn)$ with $V^{s,A}_{d,1}(\rn)$ is needed because of the use of fractional seminorms depending only on the weak derivatives of order $[s]$ in the relevant embeddings.

Embedding \eqref{intro1} can be regarded as an analog for the spaces $V^{s,A}(\rn)$  of the Morrey embedding into H\"older spaces for classical integer-order Sobolev spaces. In particular,  well-known embeddings for the fractional Sobolev spaces $V^{s,p}(\rn)$ into H\"older spaces, as in \cite{DPV, Miro, Triebel}, are recovered. Furthermore,  when $s$ is turned into an integer,
our embeddings into the optimal spaces $C^{\omega (\cdot)}(\rn)$ perfectly fit the available results for integer-order Orlicz-Sobolev spaces \cite{cianchi_ASNS, CiRa}.

As a preliminary, in Theorem \ref{necessity} we show that, because of scaling properties of the relevant (quasi-)norms in the entire space $\rn$, an embedding of the form  \eqref{intro2} is only possible if $s\in (1, n+1)\setminus \N$. Within the range of values of $s\in (0, n+1)\setminus \N$, the necessary and sufficient conditions on $A$ for the appropriate embedding \eqref{intro1}  or \eqref{intro2}  to hold read differently, according to whether $s\in (0,1)$,  $s\in (1,n)$ and $s\in (n, n+1)$. The expression of the optimal modulus of continuity $\omega$ also has a different form in these three ranges for $s$.  The corresponding results are formulated in Theorems \ref{thm:s<1}, \ref{thm:1<s<n} and \ref{thm:n<s<n+1}, respectively.

Customary techniques that are appropriate for classical fractional Sobolev spaces, such as characterizations of H\"older spaces in terms of Campanato spaces \cite{DPV}, Littlewood-Paley decompositions \cite{Triebel}, and Hardy type inequalities \cite{Miro},  fail to yield optimal conclusions when dealing with general fractional Orlicz-Sobolev spaces. Neither produces sharp results the method introduced in \cite{Garsia} for fractional-order spaces with a generalized notion of integrability, which involves rearrangements of functions and combinatorial arguments.
We refer to the paper \cite{ACPS_camp} for a discussion of the loss of information on the modulus of continuity of functions from $V^{s,A}(\rn)$ when Campanato type methods are employed.

We have thus to resort to different approaches, which depend on whether $s\in (0,1)$ or $s\in (1,n+1)$. In the former case,  we turn the original fractional problem in $\rn$ into a weighted integer-order inequality in $\mathbb R^{n+1}$. This is accomplished via an extension operator, in the spirit of ideas from \cite{MaSh2} and \cite{CaSil}. The resultant integer-order inequality in $\mathbb R^{n+1}$ is established through representation formulas and ad hoc inequalities in Orlicz spaces involving suitable measures.  In the special case of Lebesgue spaces, the relevant inequalities take a classical form. Hence, the method introduced offers a simple alternate argument for the proof of the H\"older continuity of functions from usual fractional Sobolev spaces.

By contrast, the proof for $s\in (1,n+1)$  relies upon a two-step argument. First, we apply an embedding for fractional Orlicz-Sobolev spaces of order $(s-1)$ into an Orlicz-Lorentz space, which is optimal among all rearrangement-invariant spaces. This enables us to identify the best possible degree of integrability of $\nabla u$ for any function $u \in V^{s,A}(\rn)$. Second, having reduced the initial fractional-order problem to a first-order one, we can make use of a characterization from \cite{CP-2003} to determine an estimate for the modulus of continuity of $u$. Importantly, this iteration process does not cause any loss of sharpness. The modulus of continuity obtained as a consequence of this procedure is indeed shown to be always the best possible. Incidentally, let us notice that this interesting phenomenon of sharpness preservation after iteration is already encountered when iterating Sobolev embeddings with optimal rearrangement-invariant target spaces \cite{CPS_Advances}. Such a property seems to be peculiar to rearrangement-invariant spaces: examples demonstrate that subsequent applications of embeddings which are just sharp within smaller classes of spaces, such as Lebesgue or even Orlicz spaces, need not result in a sharp final embedding in the same class.

Statements and proofs of the results outlined above are preceded by two preliminary sections. The former is devoted to basic properties of Young functions and the introduction of auxiliary functions that have a role in our embeddings. In the latter, definitions and notations concerning the function spaces appearing in this paper are recalled.

\section{Young functions}\label{youngfunct}

 A  function  $A\colon  [0,
\infty ) \to [0, \infty ]$ is called a \emph{Young function} if it is convex, non-constant, left-continuous and $A(0)=0$. Any function enjoying these properties admits the representation
\begin{equation}\label{young}
A(t) = \int _0^t a(\tau ) d\tau \qquad \text{for $t \geq 0$,}
\end{equation}
for some non-decreasing, left-continuous function $a\colon  [0, \infty )
\to [0, \infty ]$ which is neither identically equal to $0$ nor to
$\infty$.
The function
\begin{equation}\label{incr}
\frac{A(t)}t \quad \text{is non-decreasing in $(0,\infty)$.}
\end{equation}
Also,
\begin{equation}\label{kt}
kA(t) \leq A(kt) \qquad \text{for $k \geq 1$ and $t \geq 0$.}
\end{equation}
As a consequence of \eqref{kt},
\begin{equation}\label{min}
\min\{1, k\}A^{-1}(t) \leq A^{-1}(kt) \leq \max\{1, k\}A^{-1}(t) \qquad\text{for $k, t \geq 0$.}
\end{equation}
The \emph{Young conjugate} $\widetilde{A}$ of $A$  is the Young function defined as
\begin{equation}\label{Atilde}
\widetilde{A}(t) = \sup \{\tau t-A(\tau ):\,\tau \geq 0\}  \qquad \text{for $t\geq 0$.}
\end{equation}
The following formula holds:
\begin{equation}\label{youngconj}
\widetilde A(t) = \int _0^t a^{-1}(\tau ) d\tau \qquad \text{for $t \geq 0$,}
\end{equation}
where  $a^{-1}$  denotes the left-continuous (generalized) inverse of the
function $a$ appearing in \eqref{young}.
One has that
\begin{equation}\label{AAtilde}
    t \leq A^{-1}(t)\widetilde A^{-1}(t) \leq 2t \qquad \text{for $t \geq 0$.}
\end{equation}
 A  Young function $A$ is said to \emph{dominate} another Young function $B$ \emph{globally} [resp. \emph{near zero}] [resp. \emph{near infinity}]
   if there exist positive
 constants $c$ and $t_0$ such that
\begin{equation}\label{B.5bis}
B(t)\leq A(c t) \qquad \text{for $ t\geq 0$ \; [for $0\leq  t\leq t_0$] \; [for $t\geq t_0$]}.
\end{equation}
The functions $A$ and $B$ are called \emph{equivalent globally}, or \emph{near zero}, or \emph{near infinity},   if they dominate each other in the respective range of values of their argument.
We shall write
$$A \simeq B$$
to denote that
$A$ is equivalent to $B$.
\\
By contrast, the relations $$ B\lesssim A \quad \text{and} \quad  A\approx B$$ between two functions will be used to denote that $B$ is bounded by $A$ and that they are bounded by each other, respectively, up to positive multiplicative constants depending on appropriate  quantities.
\\
An optimal lower bound, in terms of a power function, for a Young function $A$ near $0$ and near infinity can be obtained by means of its lower  \emph{Matuszewska-Orlicz indices}, defined as
\begin{equation}\label{index}
   i_0(A) 	= \lim_{\lambda \to 0^+} \frac{\log {\color{black}\lambda}}{\log \Big(\liminf_{t\to0^+} \frac{A^{-1}(\lambda t)}{A^{-1}(t)} \Big)} \qquad \text{and} \qquad    i_\infty(A) 	= \lim_{\lambda \to 0^+} \frac{\log {\color{black}\lambda}}{\log \Big(\liminf_{t\to\infty} \frac{A^{-1}(\lambda t)}{A^{-1}(t)} \Big)}.
\end{equation}
If $A$ is finite-valued and vanishes only at $0$, then the following alternative expressions for $i_0(A)$ and $i_\infty(A)$ hold:
\begin{equation}\label{indexbis}
i_0(A) = \lim_{\lambda \to \infty} \frac{\log \Big(\liminf _{t\to 0^+}\frac{A(\lambda t)}{A(t)}\Big)}{\log \lambda} \qquad \text{and} \qquad   i_\infty(A) =\lim_{\lambda \to \infty} \frac{\log \Big(\liminf _{t\to \infty}\frac{A(\lambda t)}{A(t)}\Big)}{\log \lambda}.
\end{equation}
Let $q>1$ and let $A$ be a Young function. Then,
\begin{equation}\label{ibero0}
\int_0 \left(\frac{t}{A(t)}\right)^{q-1}\; dt<\infty
\quad \text{if and only if} \quad
    \int_0 \frac{\widetilde A(t)}{t^{1+q}}\, dt < \infty,
\end{equation}
and
\begin{equation}\label{iberoinf}
\int^\infty \left(\frac{t}{A(t)}\right)^{q-1}\; dt<\infty
\quad \text{if and only if} \quad
    \int^\infty \frac{\widetilde A(t)}{t^{1+q}}\, dt < \infty,
\end{equation}
see e.g. \cite[Lemma 2.3]{cianchi-ibero}.
\\ The functions associated with $A$ and $s$ according to the following definitions play a role in the statements and proofs of our embedding theorems.
\\
Given  $s\in (0, n)\setminus\N$ and a Young function $A$ such that
\begin{equation}\label{convinf}
\int^\infty \left(\frac{t}{A(t)}\right)^{\frac s{n-s}}\; dt<\infty,
\end{equation}
we
define the Young function $\EE$ as
\begin{equation}\label{\EE}
\EE(t)= t^{\frac n{n-s}}\,
\int_t^\infty \frac{ \widetilde{A}(\tau)}{\tau^{1+ \frac n{n-s}}}\; d\tau \qquad \hbox{for $t\geq 0$.}
\end{equation}
Observe that the integral on the right-hand side of \eqref{\EE} is convergent owing to assumption \eqref{convinf} and property \eqref{iberoinf}. Moreover, we call
 $\vartheta_s \colon  (0, \infty) \to (0, \infty)$ the function obeying
\begin{align}\label{202}
 \vartheta _s(r) =  \frac 1{r^{n-s}\, \EE^{-1}(r^{-n})}
\qquad \hbox{for $r>0$}.
\end{align}
iven $s\in (1, n+1)\setminus\N$ and a Young function $A$ such  that
\begin{equation}\label{conv0'}
\int_0\left(\frac{t}{A(t)}\right)^{\frac {s-1}{n-(s-1)}}\; dt<\infty,
\end{equation}
we define the Young function $\FF$ by
\begin{equation}\label{\FF}
   \FF(t)= t^{\frac{n}{n-(s-1)}} \, \int_0^t \frac{\widetilde{A}(\tau)}{\tau^{1+\frac n{n-(s-1)}}}\; d\tau \qquad \hbox{for $t\geq 0$.}
\end{equation}
The convergence of the integral on the right-hand side of equation \eqref{\FF} is guaranteed by assumption \eqref{conv0'} and property \eqref{ibero0}.
Moreover, we denote by
 $ \varrho _s \colon  (0, \infty) \to (0, \infty)$ the function defined as
\begin{align}\label{210}
 \varrho _s(r) =  \frac 1{r^{n-s}\, \FF^{-1}(r^{-n})}
\qquad \hbox{for $r>0$.}
\end{align}
Basic properties of the functions $\EE$ and $\vartheta_s$ are collected in Proposition \ref{prop:E} below; parallel properties of the functions $\FF$ and $\varrho_s$ are stated in the subsequent Proposition \ref{prop:F}.

\begin{proposition}\label{prop:E}
Let  $s\in (0,n)$ and let $A$ be a Young function. Assume that condition \eqref{convinf} is fulfilled.
Let $\EE$ and $\vartheta_s$ be defined as above. Then:
\\ {\rm (i)}  $\EE$ is a Young function.
\\ {\rm (ii)} We have that
\begin{equation}\label{dec80}
\lim_{r\to 0^+} \vartheta_s (r)=0.
\end{equation}
\\ {\rm (iii)} The function  $\vartheta_s$ is non-decreasing.
\\ {\rm (iv)}
Let $B$ be the Young function given by
 \begin{equation}\label{B}
B(t) = \widetilde \EE (t) \qquad \text{for $t\geq 0$.}
\end{equation}
Then,
\begin{equation}\label{dec14}
\vartheta_s (r) \approx r ^s B^{-1}(r^{-n}) \qquad \hbox{for $r>0$}.
\end{equation}
\\ {\rm (v)}
If
\begin{align}\label{dec112}
i_\infty(A)>\frac ns,
\end{align}
then $B\simeq A$ near infinity, and
\begin{align}\label{dec113}
\text{$\vartheta_s (r) \approx  r ^s A^{-1}(r^{-n})$ \qquad for $0<r\leq1$}.
\end{align}
\\ {\rm (vi)}
If
\begin{align}\label{dec111}
i_0(A)>\frac ns,
\end{align}
then $B\simeq A$ near $0$, and
\begin{align}\label{dec110}
\text{$\vartheta_s (r) \approx  r ^s A^{-1}(r^{-n})$ \qquad for $r\geq 1$}.
\end{align}
\end{proposition}

\begin{proof}[Proof of Proposition \ref{prop:E}] (i)
A change of variables in the integral on the right-hand side of equation \eqref{\EE} yields the alternate formula
$$
    \EE(t) = \int_1^\infty \frac{ \widetilde{A}(t\tau)}{\tau^{1+ \frac n{n-s}}}\; d\tau \qquad \hbox{for $t\geq 0$}.
    $$
This shows that $\EE$ is a Young function, since $\widetilde A$ is a Young function.
\\ (ii) The limit in  \eqref{dec80} is equivalent to $\lim_{t\to \infty}t^{\frac{n-s}n}/\EE^{-1}(t)=0$, and the latter to $\lim_{\tau\to \infty}\tau^{-\frac n{n-s}}\EE(\tau)=0$. This limit holds by the very definition of $\EE$.
\\ (iii) A change of variables shows that the non-decreasing monotonicity of the function $\vartheta _s$ is equivalent to the non-increasing monotonicity of the function $\EE(t)^{\frac {n-s}n}t^{-1}$. Since the latter agrees with $\big(\int_t^\infty \frac{ \widetilde{A}(\tau)}{\tau^{1+ \frac n{n-s}}}\; d\tau\big)^{\frac {n-s}n}$, the assertion follows.
\\ (iv) Equation \eqref{dec14} is a consequence of definition \eqref{202} and of property \eqref{AAtilde} applied to $\EE$.
\\ (v)-(vi) Equations \eqref{dec110} and \eqref{dec113} follow from  \eqref{dec14} and the fact that, under the corresponding assumptions on the Matuszewska-Orlicz  indices $i_0(A)$ and $i_\infty (A)$,  one has that $B\simeq A$ near $0$ and near infinity, respectively. For the latter assertion see e.g.  \cite[Proposition 4.1]{CianchiMusil}.
\end{proof}

\begin{proposition}\label{prop:F}
Let  $s\in (1,n+1)$ and let $A$ be a Young function. Assume that condition \eqref{conv0'} is fulfilled.
Let $\FF$ and $\varrho_s$ be defined as above.  Then:
\\ {\rm (i)} $\FF$ is a Young function.
\\ {\rm (ii)}
 If either $s\in (n, n+1)$, or $s\in (1,n)$ and assumption \eqref{convinf} is in force, then
\begin{equation}\label{dec83}
\lim_{r\to 0^+} \varrho_s (r)=0.
\end{equation}
\\ {\rm (iii)} If $s\in (n, n+1)$, then the function $\varrho_s$ is increasing.
\\ {\rm (iv)}
Let
$C$ be the Young function given by
\begin{equation}\label{C}
C(t) = \widetilde \FF (t)\qquad \text{for $t\geq 0$.}
\end{equation}
Then,
\begin{equation}\label{27}
    \varrho_s(r)\approx r^s C^{-1}(r^{-n}) \qquad \hbox{for $r>0$.}
    \end{equation}
        \\ {\rm (v)}
 If
    \begin{align}\label{dec117}
i_\infty(A)>\frac n{s-1},
\end{align}
 then $C\simeq A$ near infinity, and
\begin{align}\label{dec118}
\text{$\varrho_s (r) \approx  r ^s A^{-1}(r^{-n})$ \qquad for $0<r\leq1$.}
\end{align}
\\ {\rm (vi)} If
\begin{align}\label{dec115}
i_0(A)>\frac n{s-1},
\end{align}
 then $C\simeq A$ near $0$, and
\begin{align}\label{dec116}
\text{$\varrho_s (r) \approx  r ^s A^{-1}(r^{-n})$ \qquad for $r\geq 1$.}
\end{align}
\end{proposition}

\begin{proof}[Proof of Proposition \ref{prop:F}]
The proofs of the assertions of this proposition parallel their analogs in Proposition \ref{prop:E}. We thus limit ourselves to providing a few details on the derivation of the limit in \eqref{dec83}, which requires some additional argument. Indeed, this limit is equivalent to $\lim_{t\to \infty}\FF(t)^{\frac{n-s}n}/t=0$. If $s\in (n,n+1)$, then the latter limit holds trivially, since $\lim_{t\to \infty}\FF(t)=\infty$. Next, assume that $s\in (1,n)$ and condition  \eqref{convinf}  holds. Then, $\lim_{t\to \infty}\FF(t)^{\frac{n-s}n}/t=0$ is equivalent to
\begin{equation}\label{dec130}
\lim_{t\to \infty}\FF(t)t^{-\frac n{n-s}}=0.
\end{equation}
By property \eqref{iberoinf},  condition \eqref{convinf}  is equivalent to
$\int^\infty \frac{ \widetilde{A}(t)}{t^{1+ \frac n{n-s}}}\; dt < \infty$.
Hence, given any $\varepsilon >0$, there exists $t_\varepsilon>0$ such that
$$
\int_{t_\varepsilon}^t \frac{ \widetilde{A}(\tau)}{\tau^{1+ \frac n{n-s}}}\; d\tau < \varepsilon
\qquad \text{for
$t>t_{\varepsilon}$}
.$$
Therefore, if $t>t_{\varepsilon}$, then, inasmuch as $\frac n{n-(s-1)}- \frac n{n-s}<0$,
\begin{align}\label{dec131}
\FF(t)t^{-\frac n{n-s}} & = t^{\frac n{n-(s-1)}- \frac n{n-s}} \int^t_0 \frac{ \widetilde{A}(\tau)}{\tau^{1+ \frac n{n-(s-1)}}}\; d\tau
\\ \nonumber & =  t^{\frac n{n-(s-1)}- \frac n{n-s}} \int^{t_\varepsilon}_0 \frac{ \widetilde{A}(\tau)}{\tau^{1+ \frac n{n-(s-1)}}}\; d\tau  +  t^{\frac n{n-(s-1)}- \frac n{n-s}} \int^t_{t_\varepsilon} \frac{ \widetilde{A}(\tau)}{\tau^{1+ \frac n{n-s}}}\tau^{\frac n{n-s}- \frac n{n-(s-1)}}\; d\tau
\\ \nonumber & \leq  t^{\frac n{n-(s-1)}- \frac n{n-s}} \int^{t_\varepsilon}_0 \frac{ \widetilde{A}(\tau)}{\tau^{1+ \frac n{n-(s-1)}}}\; d\tau  +  \int^t_{t_\varepsilon} \frac{ \widetilde{A}(\tau)}{\tau^{1+ \frac n{n-s}}}\; d\tau
\\ \nonumber & \leq  t^{\frac n{n-(s-1)}- \frac n{n-s}} \int^{t_\varepsilon}_0 \frac{ \widetilde{A}(\tau)}{\tau^{1+ \frac n{n-(s-1)}}}\; d\tau  + \varepsilon.
\end{align}
Hence, $\limsup_{t\to \infty}\FF(t)t^{-\frac n{n-s}}\leq\varepsilon$, whence \eqref{dec130} follows, owing to the arbitrariness of $\varepsilon$.
\end{proof}

\section{Function spaces}\label{sec:spaces}

\subsection{Spaces of  continuous functions}
We call {\emph{modulus of continuity}} a function $$\omega\colon(0,\infty)\to (0,\infty)$$ equivalent, up to multiplicative constants, to a non-decreasing function and such that $$\lim _{r\to 0^+} \omega (r)=0.$$
By
$\mathcal{C}^{\omega(\cdot)}(\rn)$,  we denote the {\emph{space of uniformly continuous functions}} with modulus of continuity $\omega (\cdot)$ endowed with the seminorm
\begin{equation}\label{sem_mod}
\|u\|_{\mathcal{C}^{\omega(\cdot)}(\rn)} = \sup_{x,y\in \rn , x\neq y} \frac {|u(x)-u(y)|}{\omega(|x-y|)}.
\end{equation}
The space of plainly continuous functions in $\rn$ is denoted by $\mathcal{C}(\rn)$.

\subsection{Orlicz and Orlicz-Lorentz spaces}
Given a measurable set $\Omega \subset \rn$,  we define
\begin{equation}\label{M}
\mathcal{M}(\Omega)=\{u\colon \Omega \to \R : \text{$u$ is   measurable}\}.
\end{equation}
The \emph{Orlicz space} $L^A (\Omega)$, built upon a Young function
$A$,   is the Banach
space of those  functions $u\in \mathcal M(\Omega)$ for which the
 \emph{Luxemburg norm}
\begin{equation}\label{lux}
 \|u\|_{L^A(\Omega)}= \inf \left\{ \lambda >0 :  \int_{\Omega}A
\left( \frac{|u|}{\lambda} \right) dx \leq 1 \right\}\,
\end{equation}
is finite. In particular, $L^A (\Omega)= L^p (\Omega)$ if $A(t)=
t^p$ for some $p \in [1, \infty )$, and $L^A (\Omega)= L^\infty
(\Omega)$ if $A(t)=0$ for $t\in [0, 1]$ and $A(t) = \infty$ for
$t\in (1, \infty)$.
\par\noindent
A version of  H\"older's inequality in Orlicz spaces reads:
\begin{equation}\label{holder}
\int _{\Omega} |u v|\,dx \leq 2\|u\|_{L^A(\Omega)}
\|v\|_{L^{\widetilde A}(\Omega)}
\end{equation}
 for  $u \in L^A(\Omega)$ and $v\in L^{\widetilde
A}(\Omega)$. Also,
\begin{equation}\label{revholder}
\|v\|_{L^{\widetilde A}(\Omega)} \leq \sup _{u\in L^A(\Omega)} \frac{ \displaystyle\int _{\Omega} |u v|\,dx}{\|u\|_{L^A(\Omega)}} .
\end{equation}
If either $B\lesssim A$ globally, or $B\lesssim A$ near infinity and $\Omega$ has finite Lebesgue measure $|\Omega|$, then
\begin{equation}\label{normineq}
\|u \|_{L^B(\Omega)} \leq c \|u \|_{L^A(\Omega)}
\end{equation}
for $u \in L^A(\Omega)$. The constant $c$ in \eqref{normineq} agrees with the  constant appearing in inequality
\eqref{B.5bis} (with $t_0=0$) if $B\lesssim A$ globally; otherwise, it also depends on $A$ , $B$ and $|\Omega|$.
In particular, if either $A\simeq B$ globally,  or $B\simeq A$ near infinity and $|\Omega|<\infty$, then $L^A(\Omega)= L^B(\Omega)$, up to equivalent norms.
\par The family of the Orlicz-Lorentz  spaces generalizes that of the Orlicz spaces. Given  a Young function $A$ and a number $q\in \R$, the \emph{Orlicz-Lorentz space
 $L(A,q)(\Omega)$ }  consists of those functions $u \in \mathcal M(\Omega)$ making the quantity
\begin{equation}\label{aug300}
	\|u\|_{L(A, q)(\Omega)}
		= \big\|r^{-\frac{1}{q}}u^{*}(r)\big\|_{L^A(0,|\Omega|)}
\end{equation}
 finite.  The functional $\|\cdot\|_{L(A, q)(\Omega)}$ is a norm, under suitable assumptions on $A$ and $q$, and  $L(A,q)(\Omega)$ is a (non-trivial) Banach space equipped with this norm. This holds, for instance, provided that $q>1$ and
\begin{equation}\label{aug310}
\int^\infty \frac{A(t)}{t^{1+q}}\, dt < \infty\,,
\end{equation}
see \cite[Proposition 2.1]{cianchi-ibero}.

\subsection{Rearrangement-invariant spaces}

\par The Orlicz spaces and the Orlicz-Lorentz  spaces are special instances of rearrangement-invariant spaces.  The
 \emph{decreasing
rearrangement} $u^{\ast} \colon  [0,\infty) \to [0,\infty]$ of a function $u\in \mathcal M (\Omega)$ is defined as
\begin{equation}\label{u*}
u^{\ast}(r) = \inf \{t\geq 0: |\{x\in \Omega: |u(x)|>t \}|\leq r \} \qquad \text{for $r\geq 0$.}
\end{equation}
Namely, $u^{\ast}$
is the (unique) non-increasing,
right-continuous function
 equidistributed with $u$.
\\
The
\emph{Hardy-Littlewood inequality} tells us that
\begin{equation}\label{B.0}
\int_{\Omega}|uv| \,dx \leq \int_{0}^{\infty}u^{\ast}v^{\ast}\,dr
\end{equation}
for all functions $u, v \in \mathcal M(\Omega)$.

 \par\noindent A \emph{rearrangement-invariant space} is a
Banach function space $X(\Omega)$, in the sense of Luxemburg \cite[Chapter 1, Section 1]{BS},   such that
\begin{equation}\label{B.1}
 \|u\|_{X(\Omega)} = \|v \|_{X(\Omega)} \qquad \text{if $u^*=v^*$.}
 \end{equation}
Rearrangement-invariant spaces on $(0, \infty)$ are defined analogously.

\par\noindent
The \emph{representation space}  of a rearrangement-invariant space $X(\Omega)$
is  the unique rearrangement-invariant space $\overline{X}(0,|\Omega|)$ such that
\begin{equation}\label{B.3}
\|u \|_{X(\Omega)} = \|u^{\ast} \|_{\overline{X}(0,|\Omega|)}
\end{equation}
for every $u\in X(\Omega)$.
\\ A basic property  tells us that, if $X(\Omega)$ and $Y(\Omega)$ are rearrangement-invariant spaces, then
\begin{equation}\label{inclusion-embedding}
X(\Omega) \subset Y(\Omega) \qquad \text{if and only if} \qquad X(\Omega) \to Y(\Omega).
\end{equation}
 The \emph{associate space} $X^{'}(\Omega)$ of $X(\Omega)$ is the rearrangement-invariant
space of all functions in $\mathcal M(\Omega)$ for which the norm
\begin{equation}\label{B.2}
 \|v \|_{X^{'}(\Omega)} =
\sup_{u \neq 0} \frac{\int _{\Omega}|uv| dx}{\|u \|_{X(\Omega)}}
\end{equation}
is finite.
\\ The associate space of the intersection of two rearrangement-invariant spaces $X(\Omega)$ and $Y(\Omega)$ obeys:
\begin{align}\label{associate-of-cap}
  \left(X\cap Y\right)'(\Omega)=X'(\Omega) + Y'(\Omega),
\end{align}
up to equivalent norms. This is a consequence of~\cite[Theorem~3.1]{L:78} (see also~\cite[Lemma 1.12]{CNS:03}).
\\
Given any $\lambda>0$, the \textit{dilation operator} $\mathcal E_{\lambda}$, defined at a function
$f\in \M(0,\infty)$ by
\begin{equation}\label{dilation}
  \mathcal E_{\lambda} f(r)=  f(r/\lambda)\qquad  \text{for $r>0$,}
\end{equation}
is bounded on any rearrangement-invariant~space $X(0,\infty)$, with norm
not exceeding $\max\{1, 1/\lambda\}$.

\subsection{Integer-order Sobolev spaces}

Let $m \in \N$ and let $A$ be a Young function. The homogeneous $m$-th order Orlicz-Sobolev space $V^{m,A}(\rn )$ is defined as:
\begin{equation}\label{homorliczsobolev}
V^{m,A}(\rn ) = \big\{ u\in \mathcal M(\rn):\,\,\text{$u$ is $m$-times weakly differentiable and $|\nabla ^m u| \in L^A(\rn)$}\big\}.
\end{equation}
Here, $\nabla ^m u$ denotes the vector of all weak derivatives of $u$ of order $m$. If $m=1$, we also simply write $\nabla u$ instead of $\nabla^1 u$. The notation $\nabla^0 u$ has to be interpreted as $u$.
\\
Analogously, the $m$-th order Sobolev space associated with  any  rearrangement-invariant space $X(\rn)$ is given by
\begin{equation}\label{Xsobolev}
V^{m}X(\rn ) = \big\{ u \,\,\text{is $m$-times weakly differentiable in $\rn$}:\,   |\nabla ^m u| \in X(\rn)\big\}.
\end{equation}

The following result follows from  \cite[Theorem 1.3]{CP-2003}.
  Observe that theorem deals with function spaces defined on cubes, but its proof carries over to establish the present statement.

 \begin{theoremalph}\label{CP} \emph {Let $X(\rn)$ be a rearrangement-invariant space and let $\varpi_X \colon  (0,\infty) \to [0, \infty]$ be the function defined by
 \begin{equation}\label{varpi}
 \varpi (r) = \|\rho^{-1/n'}\chi_{(0,r^n)}(\rho)\|_{\overline X' (0, \infty)}\qquad \text{for $r > 0$.}
 \end{equation}
 If $\varpi $ is finite-valued and
 \begin{equation}\label{CP2}
 \lim_{r\to 0^+}\varpi (r) =0,
 \end{equation}
 then
 \begin{equation}\label{CP1}
 V^1X(\rn) \to \mathcal C^{\varpi (\cdot)}(\rn).
 \end{equation}
 Moreover,  the target space $C^{\varpi (\cdot)}(\rn)$ is optimal in \eqref{CP1} among all spaces of uniformly continuous functions.}
 \end{theoremalph}

\subsection{Fractional Orlicz-Sobolev spaces}

Given  $s\in (0,1)$ and a Young function $A$, we
denote by $J_{s,A}$ the functional defined as:
\begin{equation}\label{J}
J_{s,A}(u) = \int_{\rn} \int_{\rn}A\left(\frac{|u(x)-u(y)|}{|x-y|^s}\right)\frac{\,dx\,dy}{|x-y|^n}
\end{equation}
for $u \in \mathcal M(\rn)$.
\\ The homogeneous
 fractional Orlicz-Sobolev space $V^{s,A}(\rn)$ is given by
\begin{equation}\label{aug341}
	V^{s,A}(\rn) = \big\{u \in \mathcal M (\rn):  \,\, \text{there exists $\lambda>0$ such that $J_{s,A}(u/\lambda)<\infty$}\big\}\,.
\end{equation}
The functional $|\cdot|_{s,A, \rn}$ given by \eqref{intro0} defines a seminorm on $V^{s,A}(\rn)$. Notice that, having introduced the notation \eqref{J}, one has that
\begin{equation}\label{aug340}
|u|_{s,A, \rn}
		= \inf\left\{\lambda>0: J_{s,A}(u/\lambda)\leq 1\right\}
\end{equation}
for $u\in V^{s,A}(\rn)$.
\\ The definitions of the seminorm $|u|_{s,A, \rn}$ and of the space $V^{s,A}(\rn)$ also apply to vector-valued functions $u$,  provided that  the absolute value of $u(x)-u(y)$ is replaced with the norm  of the same expression on the right-hand side of equation \eqref{J}.
\\ The definition of $V^{s,A}(\rn)$ can be generalized to all $s\in (0, \infty) \setminus \N$  in a customary way as follows.
Denote by  $\{s\}= s-[s]$ the fractional part of $s$.
Then we set
\begin{equation}\label{aug343}
V^{s,A}(\rn ) = \big\{ u\in \mathcal M(\rn):\,\,\text{$u$ is $[s]$-times weakly differentiable and $\nabla ^{[s]}u \in V^{\{s\}, A}(\rn)$}\big\}\,.
\end{equation}
The functional
\begin{equation}\label{dec220}
\big|\nabla ^{[s]}u\big|_{\{s\},A, \rn}
\end{equation}
 defines a seminorm on the space $V^{s,A}(\rn)$.
\\ The subspace of those functions in $V^{s,A}(\rn )$ which decay near infinity, together with all derivatives, up to order $[s]$, is defined as

$$V^{s,A}_{d}(\rn)=\Big\{u\in V^{s,A}(\rn): \big|\big\{|\nabla ^h u| >t\big\}\big|<\infty \; \text{for every} \, t>0, \, \text{and} \,0\leq h\leq [s]\Big\}.$$
When $s>1$, we also need to introduce the space $V^{s,A}_{d,1}$ defined analogously to $V^{s,A}_{d}(\rn)$, save that the decay condition is only imposed starting from the first-order derivatives. Namely,
$$V^{s,A}_{d,1}(\rn)=\Big\{u\in V^{s,A}(\rn): \big|\big\{|\nabla ^h u| >t\big\}\big|<\infty \; \text{for every} \, t>0, \, \text{and} \,1\leq h\leq [s]\Big\}.$$

Unlike Lebesgue, Sobolev, and fractional Sobolev spaces, Orlicz spaces and the corresponding (fractional)  Orlicz-Sobolev spaces are not separable in general. In particular, smooth functions are not dense in
$V^{s, A}(\rn)$. Yet, any function $u \in V^{s, A}(\rn)$ can be approximated by smooth functions in modular sense, namely with respect to the functional $J_{s, A}$ defined as in \eqref{J}, for every $s\in (0, \infty)\setminus \N$ and every  finite-valued Young function $A$. This is shown in \cite[Theorem 5.1]{ACPS_inf}. Although that theorem is stated for functions from the space $V^{s,A}_d(\rn)$, it continues to hold, with the same proof, for every function in  $V^{s, A}(\rn)$. It is reproduced, for the latter space, in Theorem \ref{thm_modular} below just for $s\in (0, 1)$, the only case needed for our purposes.

\begin{theoremalph}{\rm {\bf [Modular smooth approximation]}}
\label{thm_modular}\emph{
Let $s\in (0, 1)$ and let $A$ be a finite-valued Young function. Let  $J_{s, A}$ be the functional defined as in  \eqref{J}.
Assume that  $u\in V^{s,A}(\rn)$. Then,  there exist $\lambda >0$ and a sequence $\{u_k\} \subset  V^{s,A}(\rn) \cap \mathcal{C}^\infty(\rn)$ such that
\begin{equation}\label{nov_500}
\lim_{k\to \infty} J_{s,A}\left (\frac{ u_k - u}{\lambda}\right) =0\,.
\end{equation}}
\end{theoremalph}

We conclude this section by recalling an embedding for the space $V^{s,A}_d(\R^n)$ into an Orlicz-Lorentz space, which is the optimal target space among all rearrangement-invariant spaces. The relevant embedding is crucial in one step of the proof of our results for $s>1$, and established in \cite{ACPS_emb}, where an embedding for the same space into an optimal Orlicz target space can also be found.

The embedding in question is stated in Theorem \ref{a} and holds for any $s\in (0,n)\setminus \N$ and any Young function $A$ satisfying the necessary condition:
\begin{equation}\label{conv0n}
\int_0 \left(\frac{t}{A(t)}\right)^{\frac s{n-s}}\; dt<\infty.
\end{equation}
It takes a different form, according to whether condition \eqref{convinf}, or its converse
\begin{equation}\label{divinf}
\int^\infty \left(\frac{t}{A(t)}\right)^{\frac s{n-s}}\; dt=\infty,
\end{equation}
is in force.
\\
The  target space in the embedding in question is the Orlicz-Lorentz space $L({\widehat A},\frac{n}{s})(\R^n)$ whose norm is defined as in \eqref{aug300}, where the
 Young function $\widehat A$ is given by
\begin{equation}\label{E:1}
	\widehat A (t)=\int_0^t\widehat a (\tau)\,d\tau\qquad\text{for $t\geq 0$},
\end{equation}
and
\begin{equation}\label{E:2}
	{\widehat a\,}^{-1}(r) = \left(\int_{a^{-1}(r)}^{\infty}
		\left(\int_0^t\left(\frac{1}{a(\rho)}\right)^{\frac{s}{n-s}}\,d\rho\right)^{-\frac{n}{s}}\frac{dt}{a(t)^{\frac{n}{n-s}}}
				\right)^{\frac{s}{s-n}}
					\qquad\text{for $r\ge0$}.
\end{equation}

\begin{theoremalph}{\rm{\bf [Optimal rearrangement-invariant target space]}}\label{a}
\emph{Let $s\in (0,n) \setminus \mathbb N$.  Assume that $A$ fulfills condition \eqref{conv0n}.
Let ${\widehat A}$ be the Young function defined as in \eqref{E:1}--\eqref{E:2}, and let $L({\widehat A},\frac{n}{s})(\R^n)$ be the Orlicz-Lorentz~space equipped with the norm defined as in \eqref{aug300}.
\\
{\rm (i)} Assume that condition \eqref{divinf} holds.
Then,
\begin{equation}\label{E:30hemb}
V^{s,A}_d(\R^n) \to L({\widehat A},\tfrac{n}{s})(\R^n),
\end{equation}
and
 there exists a~constant  $c=c(n,s)$ such that
\begin{equation}\label{E:30h}
	\|u\|_{L({\widehat A},\frac{n}{s})(\R^n)}
		\le c \,\big|\nabla ^{[s]}u\big|_{\{s\},A, \R^n}
\end{equation}
for every function $u \in V^{s,A}_d(\R^n)$.
\\ {\rm (ii)} Assume that condition \eqref{convinf} holds. Then,
\begin{equation}\label{E:30hemb'}
V^{s,A}_d(\R^n) \to \big(L^\infty \cap L({\widehat A},\tfrac{n}{s})\big)(\R^n)
\end{equation}
and
 there exists a~constant   $c=c(n,s, A)$ such that
\begin{equation}\label{E:30h'}
	\|u\|_{\big(L^\infty \cap L({\widehat A},\tfrac{n}{s})\big)(\R^n)}
		\le c \,\big|\nabla ^{[s]}u\big|_{\{s\},A, \R^n}
\end{equation}
for every function $u \in V^{s,A}_d(\R^n)$.
\\
Moreover, $L({\widehat A},\frac{n}{s})(\R^n)$ and $\big(L^\infty \cap L({\widehat A},\tfrac{n}{s})\big)(\R^n)$ are  the optimal target spaces in inequalities \eqref{E:30h} and \eqref{E:30h'}, respectively, among all rearrangement-invariant spaces.}
\end{theoremalph}


\section{Main results}\label{S:main}

Our embeddings into target spaces of uniformly continuous functions $\mathcal C^{\omega (\cdot)}(\rn)$ are formulated in terms of the seminorm \eqref{sem_mod}. Such a seminorm is invariant under perturbation of functions by additive constants, but not by polynomials of degree  $\geq 1$. On the other hand, the seminorm given by \eqref{dec220} is invariant under the addition of any polynomial of degree $\leq [s]$.
Therefore, if $s>1$, it is natural to consider  the subspace  $V^{s,A}_{d,1}(\rn)$ of $V^{s,A}(\rn)$ as a domain space for
embeddings into $\mathcal C^{\omega (\cdot)}(\rn)$. This way, polynomials of degree between $1$ and $[s]$ are ruled out from the class of trial functions.

Our first result tells us that no embedding of this form can hold if $s>n+1$. This is due to the behavior under dilations in $\rn$ of the seminorms involved in the embeddings.

\begin{theorem}\label{necessity}
Let $s\in (1,\infty)\setminus \N$ and let $A$ be a Young function.
Assume that
\begin{equation}\label{nov100}
    V^{s, A}_{d,1}(\rn)\to \mathcal{C}^{\omega (\cdot)}(\rn)
\end{equation}
 for some modulus of continuity $\omega$.
Then, $s<n+1$.
\end{theorem}

In the admissible range $(0,n+1)\setminus \N$ for the smoothness parameter $s$, the criterion for the space $V^{s,A}(\rn)$ to be embedded into a space of uniformly continuous functions and the corresponding optimal modulus of continuity $\sigma_s$ take a distinct form in the three sub-intervals $(0,1)$, $(1,n)$, and $(n, n+1)$. Plainly, only the asymptotic behaviors of $\sigma_s$ near zero and near infinity are relevant. As will be clear from our statements, these behaviors in turn only depend on the  behaviors (in the sense of the relation $\lq\lq \simeq"$  between Young functions) of $A$ near infinity and near zero, respectively.

When $s\in (0,1)$, the single condition \eqref{convinf} characterizes the validity of the embedding in question and the optimal target space is associated with the modulus of continuity given by the function $\vartheta_s$ defined as in  \eqref{202}.

\begin{theorem}{\rm{\bf [Case  $s\in (0,1)$]}}\label{thm:s<1}
Let $s\in (0,1)$ and let $A$ be a Young function satisfying condition \eqref{convinf}.
Set
\begin{align}\label{july22}
\sigma_s(r)= \vartheta_s(r) \qquad \text{for $r > 0$,}
\end{align}
where $\vartheta_s$ is the function defined by \eqref{202}.
Then, $\sigma_s$ is a modulus of continuity, and
\begin{equation}\label{23s<1}
    V^{s, A}(\rn)\to \mathcal{C}^{\sigma_s (\cdot)}(\rn).
\end{equation}
Namely,
 there exists a  constant $c$ depending on $n$ and $s$ such that
\begin{equation}\label{23's<1}
   |u(x)-u(y)|\leq c \,\sigma_s(|x-y|)\, |u|_{s, A, \rn}
\end{equation}
for $u\in V^{s, A}(\rn)$.
\\
The result is sharp, in the sense that if there exists a modulus of continuity $\omega$ such that
 $V^{s, A}(\rn)\to \mathcal{C}^{\omega (\cdot)}(\rn)$, then condition  \eqref{convinf}  holds, and $ \mathcal{C}^{\sigma_s (\cdot)}(\rn)\to  \mathcal{C}^{\omega (\cdot)}(\rn)$.
\end{theorem}

\begin{remark}\label{integrals<1}{\rm One step in the proof of Theorem \ref{thm:s<1} consists in showing the following alternate modular version of inequality \eqref{23's<1}:
\begin{equation}\label{paseky1}
|u(x)-u(y)| \leq c|x-y|^s B^{-1}\bigg(\frac 1{|x-y|^n} \int_{\rn} \int_{\rn} A\left(\frac{|u(z) - u(w)|}{|z-w|^s}\right)\; \frac{dz \, dw}{|z-w|^n}\bigg)
\end{equation}
 for every $u\in V^{s, A}(\rn)$ and $x,y\in\rn$. Here, $c$ is the same constant as in \eqref{23's<1} and $B$ is the Young function defined by \eqref{B}.
 This inequality can be handy in analyzing solutions to nonlocal equations and variational problems.}
\end{remark}

The characterization of the embeddings in question
for values of $s\in (1,n)$ requires that assumption  \eqref{convinf} be coupled with \eqref{conv0'}. Moreover, the optimal modulus of continuity is described in terms of both functions $\vartheta_s$ and $\varrho_s$, defined by equations \eqref{202} and \eqref{210}, and depends on whether
\begin{equation}\label{july30}
   \int^\infty \left(\frac t{A(t)}\right)^{\frac{s-1}{n-(s-1)}} \; dt =\infty,
\end{equation}
or
\begin{equation}\label{24}
   \int^\infty \left(\frac t{A(t)}\right)^{\frac{s-1}{n-(s-1)}} \; dt <\infty.
\end{equation}

\begin{theorem}{\rm{\bf [Case  $s\in (1,n)$]}}\label{thm:1<s<n}
Let $s\in (1,n)\setminus \N$ and let $A$ be a Young function satisfying conditions \eqref{convinf} and \eqref{conv0'}. Let $\vartheta_s$ and $\varrho_s$ be the functions defined by  \eqref{202} and \eqref{210}.
\\
  {\rm(i)} If
 \eqref{july30} holds, set
\begin{align}\label{july60}
    \sigma_s (r)=
    \vartheta_s(r) + \varrho_s(r) \qquad \hbox{for $r>0$}.
\end{align}
  {\rm(ii)} If \eqref{24} holds, set
\begin{align}\label{july61}
    \sigma_s (r)\approx \begin{cases}
    r & \quad \text{if\, $ 0< r<1$}
    \\
    \vartheta_s(r) + \varrho_s(r) & \quad \text{if\, $r\geq 1$.}
    \end{cases}
\end{align}
Then, $\sigma_s $ is a modulus of continuity, and
\begin{equation}\label{231<s<n}
    V^{s, A}_{d,1}(\rn)\to \mathcal{C}^{\sigma_s (\cdot)}(\rn).
\end{equation}
Namely, there exists a  constant $c$ depending on $n$, $s$ and $A$ such that
\begin{equation}\label{23'1<s<n}
   |u(x)-u(y)|\leq c \,\sigma_s(|x-y|)\, \big|\nabla ^{[s]}u\big|_{\{s\}, A, \rn}
\end{equation}
for $u\in V^{s, A}_{d,1}(\rn)$.
\\
The result is sharp, in the sense that if there exists a modulus of continuity $\omega$ such that
 $V^{s, A}_{d,1}(\rn)\to \mathcal{C}^{\omega (\cdot)}(\rn)$, then conditions  \eqref{convinf} and \eqref{conv0'} hold, and $ \mathcal{C}^{\sigma_s (\cdot)}(\rn)\to  \mathcal{C}^{\omega (\cdot)}(\rn)$.
\end{theorem}

The following local Lipschitz continuity result is a straightforward consequence of Theorem \ref{thm:1<s<n}.

\begin{corollary}\label{lip} Let $s\in (1,n)\setminus \N$ and let $A$ be a Young function satisfying conditions \eqref{conv0'} and \eqref{24}.
If $u \in V^{s,A}_{d,1}(\rn)$, then $u$ is locally Lipschitz continuous.
\end{corollary}

In the range of values $s\in (n,n+1)$, the existence of an embedding into a space of uniformly continuous functions again depends on a unique condition, which is now  \eqref{conv0'}. The expression of the optimal modulus of continuity $\sigma_s$ only involves the function $\varrho_s$, but takes a different form according to whether condition \eqref{july30} or \eqref{24} holds.

\begin{theorem}{\rm{\bf [Case  $s\in (n,n+1)$]}}\label{thm:n<s<n+1}
Let $s\in (n, n+1)\setminus \N$ and  let $A$ be a Young function satisfying condition  \eqref{conv0'}.  Let  $\varrho_s$ be the function defined by  \eqref{210}
\\
 {\rm(i)} If
 \eqref{july30} holds, set
\begin{align}\label{july62}
    \sigma_s (r)=
    \varrho_s(r)
    \qquad \hbox{for $r>0$}.
\end{align}
 {\rm(ii)} If \eqref{24} holds, set
\begin{align}\label{july63}
    \sigma_s (r)\approx \begin{cases}
    r & \quad \text{if\, $0<r<1$}
    \\
   \varrho_s(r) & \quad \text{if\, $r\geq 1$.}
    \end{cases}
\end{align}
Then,  $ \sigma_s $ is a modulus of continuity, and
\begin{equation}\label{23n<s<n+1}
    V^{s, A}_{d,1}(\rn)\to \mathcal{C}^{\sigma_s (\cdot)}(\rn).
\end{equation}
Namely, there exists a  constant $c$ depending on $n$, $s$ and $A$ such that
\begin{equation}\label{23'n<s<n+1}
   |u(x)-u(y)|\leq c \,\sigma_s(|x-y|)\, \big|\nabla ^{[s]}u\big|_{\{s\}, A, \rn}
\end{equation}
for $u\in V^{s, A}_{d,1}(\rn)$.
\\
The result is sharp, in the sense that if there exists a modulus of continuity $\omega$ such that
 $V^{s, A}_{d,1}(\rn)\to \mathcal{C}^{\omega (\cdot)}(\rn)$, then condition    \eqref{conv0'} holds, and $ \mathcal{C}^{\sigma_s (\cdot)}(\rn)\to  \mathcal{C}^{\omega (\cdot)}(\rn)$.
\end{theorem}

The criterion for local Lipschitz continuity ensuing from Theorem \ref{thm:n<s<n+1} is stated in the next corollary.

\begin{corollary}\label{lip1} Let $s\in (n, n+1)\setminus \N$ and let $A$ be a Young function satisfying conditions \eqref{conv0'} and \eqref{24}.
If $u \in V^{s,A}_{d,1}(\rn)$, then $u$ is locally Lipschitz continuous.
\end{corollary}

We conclude this section by applications of our results to a couple of special families of fractional Orlicz-Sobolev spaces.
Though simple, these examples
give the flavor of the variety of moduli of continuity that can surface when enlarging the class of domain spaces from the standard Sobolev to the Orlicz-Sobolev fractional spaces.

\begin{example}{\rm{\bf [Power-times-logarithm type Young functions]}}\label{ex1}
{\rm
Consider  fractional Orlicz-Sobolev spaces built upon Young functions $A$
such that
\begin{equation}\label{E:young-domain}
    A(t)\simeq
        \begin{cases}
            t^{p_0} \left(\log (1+ \frac 1t)\right)^{\alpha_0} & \quad \text{near zero}
                \\
            t^p  \left(\log (1+t)\right)^\alpha & \quad \text{near infinity.}
        \end{cases}
\end{equation}
The behaviors in \eqref{E:young-domain} are admissible for a Young function, provided that either $p_0>1$ and $\alpha_0 \in \R$, or $p_0=1$ and $\alpha_0 \leq 0$, and either $p>1$ and $\alpha \in \R$ or $p=1$ and $\alpha \geq 0$.

 \smallskip\par\noindent
 \emph{Case $s\in (0,1)$}.
 By Theorem \ref{thm:s<1},  the optimal space $\mathcal{C}^{\sigma_s(\cdot)}(\rn)$ in the embedding
\begin{equation*}
    V^{s, A}(\rn) \to \mathcal{C}^{\sigma_s(\cdot)}(\rn)
\end{equation*}
fulfills
\begin{equation}\label{vartheta_final}
\sigma_s(r)\approx\begin{cases}
    \left(\log\left( 1+ \frac 1r\right)\right)^{1-\frac{s(\alpha+1)}{n}}&\quad\hbox{if $p= \frac ns , \alpha> \frac ns -1$}
    \\
    \\
     r^{s-\frac np}\, \left(\log\left(1+\frac 1r\right)\right)^{-\frac \alpha p}&\quad\hbox{if $p>\frac ns , \alpha\in \R$}
    \end{cases}
    \qquad \qquad\text{as $r\to 0 ,$}
\end{equation}
and
\begin{equation}\label{vartheta_2}
    \sigma_s(r)\approx\begin{cases}
    1 &\quad\hbox{if }
    \begin{cases}
    p_0=1, \alpha_0\leq 0
    \\
    1<p_0<\frac ns , \alpha_0\in \R
    \\
    p_0= \frac ns ,  \alpha_0 > \frac ns -1
    \end{cases}
    \\
    \\
     \left(\log\left(\log( 1+ r)\right)\right)^{1-\frac{s}n}&\quad\hbox{if $p_0= \frac ns ,\alpha_0=\frac ns -1$}
    \\
    \\
    \left(\log\left( 1+ r\right)\right)^{1-\frac{s(\alpha_0+1)}{n}}&\quad\hbox{if $p_0= \frac ns ,\alpha_0<\frac ns -1$}
    \\
    \\
    r^{s-\frac n{p_0}}\, \left(\log\left(1+r\right)\right)^{-\frac {\alpha_0} {p_0}}&\quad\hbox{if $p_0>\frac ns ,\alpha\in \R$}
    \end{cases}\qquad\qquad\text{as $r\to \infty .$}
\end{equation}
On the other hand, the necessary assumption \eqref{convinf} in Theorem \ref{thm:s<1} is not fulfilled for the values of the parameters $\alpha, p, \alpha_0, p_0$ not included in formulas \eqref{vartheta_final} and \eqref{vartheta_2}. Hence, no embedding of the form
\eqref{intro1} holds for those values.

\smallskip
\par
\noindent
 \emph{Case $s\in (1,n)$}.
By Theorem \ref{thm:1<s<n}},
\begin{equation}\label{600}
    V^{s, A}_{d,1} (\rn)\to \mathcal{C}^{\sigma_s(\cdot)}(\rn),
\end{equation}
where the optimal modulus of continuity $\sigma_s$ satisfies:
\begin{align}\label{sigma_100}
\sigma_s(r)\approx
\begin{cases}
\left(\log\left(1 + \frac 1r\right)\right)^{1-\frac {s(\alpha +1)}n}   &\quad \hbox{if $p=\frac ns, \alpha> \frac ns -1$}
\\
\\
 r^{s-\frac np} \left(\log\left(1 + \frac 1r\right)\right)^{-\frac {\alpha}p}   &\quad \hbox{if $\frac ns<p<\frac n{s-1}, \alpha\in \R,$}
  \\
 \\
 r\, \left(\log\left(1+\frac 1r\right)\right)^{1-\frac{(s-1)(\alpha+1)}{n}}
 &\quad \hbox{if  $p=\frac {n}{s-1}$ and $\alpha < \frac n{s-1} -1$}
 \\
 \\
r\, \left(\log\left(\log\left(1+\frac 1{r}\right)\right)\right)^{1-\frac{(s-1)}{n}}
 &\quad \hbox{if  $p=\frac {n}{s-1}$ and $\alpha =  \frac n{s-1} -1$}
 \\
 \\r &\quad \hbox{if} \begin{cases}
 p> \frac n{s-1}
 \\
 p=\frac {n}{s-1}$ and $\alpha >  \frac n{s-1} -1
\end{cases}
\end{cases}
\qquad\text{as $r\to 0 ,$}
\end{align}
and
\begin{equation}
    \label{583}
    \sigma_s(r)\approx \begin{cases}
    1
     &\quad \hbox{if}
    \begin{cases}
    p_0=1, \alpha_0\leq 0
    \\
    1<p_0<\frac ns , \alpha_0\in \R
    \\
     p_0= \frac ns ,  \alpha_0 > \frac ns -1
    \end{cases}
     \\
     \\
      \left(\log\left(\log( 1+ r)\right)\right)^{1-\frac{s}n}
     &\quad\hbox{if $p_0= \frac ns, \alpha_0=\frac ns -1$}
    \\
    \\
     \left(\log\left( 1+ r\right)\right)^{1 -\frac{s(\alpha_0+1)}{n}}
      &\quad\hbox{if $p_0= \frac ns ,\alpha_0<\frac ns -1$}
       \\
    \\
    r^{s-\frac n{p_0}}\, \left(\log\left(1+r\right)\right)^{-\frac {\alpha_0} {p_0}}
    &\quad\hbox{if  $\frac ns <p_0< \frac {n}{s-1}$ ,$\alpha_0\in \R$}
    \\
    \\
r \left(\log \left(1+r\right)\right)^{1-
\frac{(1+\alpha_0)(s-1)}{n}}
    &\quad\hbox{if $p_0=\frac {n}{s-1} , \alpha_0 >
    \frac {n}{s-1} -1$}
       \end{cases}
       \qquad\text{as $r\to \infty.$}
\end{equation}
The necessary assumptions \eqref{convinf} and \eqref{conv0'} in Theorem \ref{thm:1<s<n} are not fulfilled for the values of the parameters $\alpha, p, \alpha_0, p_0$ not included in formulas \eqref{sigma_100} and \eqref{583}. Consequently, no embedding of the form
\eqref{intro2} holds for those values.

\smallskip
\par\noindent
\emph{Case $s\in (n,n+1)$}.
Theorem \ref{thm:n<s<n+1} tells us that  the optimal modulus of continuity in
\begin{equation*}
    V^{s, A}_{d,1}(\rn)\to \mathcal{C}^{\sigma_s (\cdot)}(\rn),
\end{equation*}
obeys:
\begin{equation}
    \label{561}
\sigma_s(r)\approx
\begin{cases}
r^{s-\frac{n}{p}} \left(\log \left(1+\frac 1r\right)\right)^{-\frac{\alpha}{p}} &\qquad\hbox{if }
\begin{cases}
    p=1, \alpha\geq  0
    \\
   1<p<\frac{n}{s-1}, \alpha\in \R
    \end{cases}
\\
\\
r\, \left(\log\left(1
+\frac 1r\right)\right)^{1 - \frac{(s-1)(\alpha+1)}{n}}&\qquad\hbox{if $p= \frac n{s-1} , \alpha< \frac n{s-1} -1$}
\\
\\
r\, \left(\log\left(\log\left(1+\frac 1{r}\right)\right)\right)^{1-\frac{(s-1)}{n}}&\qquad\hbox{if $p= \frac n{s-1} , \alpha=\frac n{s-1} -1$}
\\
\\
r &\qquad\hbox{if }
\begin{cases}p= \frac n{s-1} , \alpha> \frac n{s-1} -1
    \\
p>\frac{n}{s-1}, \alpha\in\R
\end{cases}
\end{cases}
\qquad\text{as $r\to 0$,}
\end{equation}
and
\begin{align}\label{589}
    \sigma_s (r)\approx
    \begin{cases}
  r^{s- \frac n{p_0}}\, \left(\log\left(1+r\right)\right)^{-\frac{\alpha_0}{p_0}} &\quad
    \hbox{if }
 \begin{cases} p_0=1, \alpha_0\leq 0
\\
1<p_0<\frac n{s-1}, \alpha_0\in\R
\end{cases}
    \\
    \\
    r \left(\log\left(1+r\right)\right)^{1-\frac{(s-1)(\alpha_0 +1)}{n}} &\quad \text{if\, $p_0=\frac{n}{s-1}$, $\alpha_0 > \frac{n}{s-1} -1$}
     \end{cases}
     \qquad \text{as $r\to \infty$.}
\end{align}
The necessary assumption \eqref{conv0'} in Theorem \ref{thm:n<s<n+1} fails for the values of the parameters $\alpha, p, \alpha_0, p_0$ which are missing in formulas \eqref{561} and \eqref{589}. Therefore, no embedding of the form
\eqref{intro2} can hold for those values.
\end{example}

\medskip

\begin{example}{\rm{\bf [Exponential type Young functions]}}\label{ex2}
{\rm
Consider  fractional Orlicz-Sobolev spaces associated with Young functions obeying:
\begin{equation}\label{E:young_exp_-domain}
    A(t)\simeq
        \begin{cases}
           e^{-t^{\frac{1}{\gamma_0}}}& \quad \text{near zero}
                \\
           e^{t^\gamma} & \quad \text{near infinity,}
        \end{cases}
\end{equation}
 where $\gamma_0 <0 $ and $\gamma >0 $.
\\ The necessary  condition \eqref{conv0'} in Theorems \ref{thm:1<s<n} and \ref{thm:n<s<n+1} does not hold, whatever $\gamma_0$ is.
Therefore, if $s\in (1,n+1)$, no embedding of the form \eqref{intro2} can hold for Young functions $A$ as in \eqref{E:young_exp_-domain}.
\\ We shall thus focus on the case when $s\in (0,1)$. An application of
 Theorem \ref{thm:s<1}, whose assumption \eqref{convinf} is satisfied for every $\gamma >0$, yields
\begin{equation*}
    V^{s, A}(\rn) \to \mathcal{C}^{\sigma_s(\cdot)}(\rn),
\end{equation*}
where the optimal modulus of continuity $\sigma_s$ fulfills:
$$\sigma_s(r) \approx \begin{cases}
r^s \left(\log\left(1+\frac 1r\right)\right)^{\frac 1{\gamma}} &\qquad \hbox{as $r\to 0$}
\\
\\
r^s \left(\log\left(1+r\right)\right)^{\gamma_0} &\qquad \hbox{as $r\to \infty$}.
\end{cases}
$$
}
\end{example}

\section{Proofs of the main results}\label{proofs}

This section is devoted to proofs of Theorems \ref{necessity}, \ref{thm:s<1}, \ref{thm:1<s<n} and \ref{thm:n<s<n+1}.  Each of them is the subject of one of the following subsections.

\subsection{Proof of Theorem \ref{necessity}}
The argument of the proof of Theorem \ref{necessity} rests upon an appropriate scaling of suitably chosen trial functions in an embedding of the form \eqref{nov100}, coupled with basic properties of Young functions.

\begin{proof}[Proof of Theorem \ref{necessity}]  Assume, by contradiction, that embedding \eqref{nov100} holds for some $s\in (n+1, \infty) \setminus \N$.
Let $\xi \in C^\infty_0(\rn)$  be  a nonnegative function such that $\nabla \xi (0)\neq 0$. For each $k \in \N$,   consider the function $u_k \colon  \rn \to \mathbb  R$ defined as
\begin{equation}\label{100a}
u_k (x)= k^{s-n} \xi \Big(\frac x k\Big) \qquad \hbox{for}\;\; x\in \rn\,.
\end{equation}
Since $u_k \in C^\infty_0(\rn)$, we have that
$$|\{ |\nabla ^k u_k|> t\}| < \infty \qquad \text{ for  $t>0$,}$$
 for  $k=0, 1, \dots, [s]$.
\\
By \cite[Inequality (6.2)]{ACPS_inf},
 there exists a constant $c$, independent of $k$, such that
\begin{equation}\label{may30}
|\nabla ^{[s]}u_k|_{\{s\}, A, \rn} \leq c.
\end{equation}
Embedding \eqref{nov100} and inequality \eqref{may30} imply that there exists a constant $c$ such that
\begin{equation}\label{nov102}
\|u_k\|_{\mathcal C^{\omega(\cdot)}(\rn)} \leq c
\end{equation}
for every $k\in \N$.
Therefore,
    \begin{equation}\label{nov103}
   |u_k(x)-u_k(0)|\leq c \,\omega (|x|) \qquad \text{for $x \in \rn$.}
    \end{equation}
Since we are assuming that   $\nabla \xi (0)\neq 0$, we may choose $x = \frac{\nabla \xi (0)}{|\nabla \xi (0)|}$ in \eqref{nov103}.
We have that
\begin{equation*}\label{nov104}
    \xi(x) = \xi(0) + x \cdot \nabla \xi (0) + \mathcal O (x^2) \qquad \text{as $x\to 0$.}
\end{equation*}
Thus,
\begin{align}\label{nov105}
  c&\geq  \frac{\big|u_k\big(\frac{\nabla \xi (0)}{|\nabla \xi (0)|}\big) - u_k(0)\big|}{\omega(1)} = k^{s-n} \frac{\big|\xi\big(\frac{\nabla \xi (0)}{k|\nabla \xi (0)|}\big) - \xi(0)\big|}{\omega(1)}\\ \nonumber & = k^{s-n} \frac{\big|\frac{|\nabla \xi (0)|}{k} + \mathcal O (j^{-2})\big|}{\omega(1)} \geq k^{s-n-1} \frac{|\nabla \xi (0)|}{\omega(1)} +\mathcal O \big(k^{s-n-2\big)}.
\end{align}
Letting $k\to \infty$ in \eqref{nov105} yields a contradiction.
\end{proof}

\subsection{Proof of Theorem \ref{thm:s<1}: case $s\in (0,1)$}

As mentioned in Section \ref{intro}, our approach to Theorem \ref{thm:s<1} starts from an extension argument, and differs from those already exploited in the literature in the available proofs of the H\"older continuity of fractional Sobolev functions. Our original attempts consisted of adapting those proofs to the Orlicz realm. However, all of them
only produce results that are, in general, weaker than the anticipated ones. This is especially apparent in borderline situations, where stronger assumptions on $s$ and $A$ than in Theorem \ref{thm:s<1} are required for the space $V^{s,A}(\rn)$ to be embedded into a space of uniformly continuous functions, and non-optimal moduli of continuity are produced.

Consider, for instance, Young functions $A$ as in Example \ref{ex1}, namely obeying \eqref{E:young-domain}. Customary arguments as in \cite{DPV}, or \cite{Miro}, or yet \cite{Triebel} enable one to conclude that, when $p=\frac ns$, functions from $V^{s,A}(\rn)$ are continuous only if $\alpha > \frac ns$. By contrast, an application of Theorem  \ref{thm:s<1} just requires that $\alpha > \frac ns -1$. Furthermore, those methods yield a modulus of continuity which behaves like $\log (1+\frac 1r)^{1-\frac{\alpha s}n}$ as $r\to 0^+$, which is not sharp, as demonstrated by equation \eqref{vartheta_final}.

\begin{proof}[Proof of Theorem \ref{thm:s<1}] The fact that $\sigma_s$ is a modulus of continuity is a consequence of properties (ii) and (iii) of Proposition \ref{prop:E}.
Let $u\in V^{s, A}(\rn)$. Assume also, for the time being, that $u$ is locally H\"older continuous. Let $\psi \colon  \rn \to [0, \infty)$ be the function defined as
\begin{equation*}
	\psi(y)=\frac{(n+1)} {\omega_n}(1-|y|)_+ \qquad \text{for $y \in \rn$,}
\end{equation*}
where $\omega_n$ denotes the Lebesgue measure of the unit ball in $\rn$ and the subscript \lq\lq$+$" stands for positive part, and let $U\colon \rn \times \R \to \R$ be the function defined by
\begin{equation}\label{E:15}
	U(x,t)=\int_{\R^n}\psi(y)u(x+ty)\,dy\qquad\text{for $(x,t)\in\R^n\times \R$.}
\end{equation}
Since we are currently assuming that $u$ is locally H\"older continuous, we have that   $U$ is also H\"older continuous and, in particular,
\begin{align}
    u(x) = \lim_{t\to 0}U(x,t) = U(x,0) \qquad \text{for $x\in \rn$.}
\end{align}
For the same reason, the classical bound
\begin{align}\label{E:16}
	|\nabla U(x,t)|
		 \le \frac{(n+1)(n+2)}{|t|\,\omega_n}\int_{\{|y|<1\}} |u(x+ty)-u(x)|\,dy
			\qquad\text{for a.e. $(x,t)\in\R^n\times\R$}
\end{align}
ensures that $U\in W^{1,1}_{\rm loc}(\rn \times \R)$.
One can show that
\begin{equation}\label{2}
\int_{\R} \int_{\rn} A\left( |t|^{1-s} |\nabla U(x,t)|\right) \;  \frac {dx\,dt}{|t|} \leq   \int_{\rn} \int_{\rn} A\left(c\frac{|u(x) - u(y)|}{|x-y|^s}\right)\; \frac{dx \, dy}{|x-y|^n}
\end{equation}
for some constant $c$ depending on $n$, see \cite[Inequality (5.15)]{ACPS_emb}.
\\
Now,  observe that, if   $v\in W^{1,1}_{\operatorname{loc}}(\rn)\cap \mathcal C(\rn)$, then
\begin{equation}\label{1}
|v(x) - v(y)|\leq c\, \left(\int_{B_\delta (x)} \frac{|\nabla v(z)|}{|x-z|^{n-1}}\; dz +
\int_{B_\delta (y)} \frac{|\nabla v(z)|}{|y-z|^{n-1}}\; dz \right) \qquad\text{for  $x,y\in\R^n$,}
\end{equation}
 for some constant $c=c(n)$,
where $\delta = |x-y|$ and $B_\delta (x)$ denotes the ball centered at $x$, with radius $\delta$.  Inequality \eqref{1} is a consequence of the following chain:
\begin{align}\label{jan1}
    |v(x) - v(y)|&\leq \frac{c}{\delta^n}\int_{B_{\delta/2}(\frac {x+y}2)}|v(x)-v(z)| + |v(z)-v(y)|\, dz
    \\ \nonumber &\leq c'\int_{B_{\delta/2}(\frac {x+y}2)}\frac{|\nabla v(z)|}{|x-z|^{n-1}}+ \frac{|\nabla v(z)|}{|y-z|^{n-1}}\, dz
    \\ \nonumber &\leq
    c'\, \left(\int_{B_\delta (x)} \frac{|\nabla v(z)|}{|x-z|^{n-1}}\; dz +
\int_{B_\delta (y)} \frac{|\nabla v(z)|}{|y-z|^{n-1}}\; dz \right) \qquad\text{for $x,y\in\R^n$,}
    \end{align}
    and for some constants $c=c(n)$ and $c'=c'(n)$,
    where the second inequality holds owing to \cite[Lemma 1, Section 4.5.2]{EG}. Notice that, although this lemma is stated for smooth functions, it continues to hold under our assumptions on  $v$.
\\
Thus, if we fix $x, y \in \rn$ and set $\delta=|x-y|$,  then inequality \eqref{1} applied (with $\rn$ replaced by $\mathbb R^{n+1}$) to the function $U$ yields:
\begin{multline}
\label{3}
|u(x) - u(y)|=|U(x, 0) - U(y, 0)| \\ \leq c\, \bigg(\int_{|z-x|^2 + t^2 \leq \delta ^2} \frac{|\nabla U(z, t)|}{(|x-z|^2 +t^2)^{\frac n2}}\; dz\, dt
  +
\int_{|z-y|^2 + t^2 \leq \delta ^2} \frac{|\nabla U(z, t)|}{(|y-z|^2 + t^2)^{\frac n2}}\; dz\, dt \bigg).
\end{multline}
Let us focus on the former integral on the rightmost side of equation \eqref{3}, since the latter can be handled analogously.
An application of H\"{o}lder's inequality \eqref{holder} with respect to the measure
$d\mu =   \frac{dz\,dt}{|t|}$
tells us that
\begin{align}\label{4}
\int_{B_\delta (x,0)} \frac {|\nabla U(z,t)|}{(|x-z|^2 + t^2)^{\frac n2}} \; dz \, dt &=
   \int_{B_\delta (x,0)} \frac {|t|^{1-s} |\nabla U(z,t)| |t|^s}{(|x-z|^2 + t^2)^{\frac n2}}\; d\mu
   \\ &
   \leq 2\, \Big \| |\nabla U(z,t)| \, |t|^{1-s}\Big\|_{L^A (d\mu)} \bigg \| \frac{\chi_{B_\delta (x,0)} (z,t) \, |t|^s}{(|x-z|^2 + t^2)^{\frac n2}}\bigg\|_{L^{\widetilde{A}}(d\mu)}.\nonumber
\end{align}
Here, and in what follows, $\chi_\Omega$ denotes the characteristic function of a set $\Omega \subset \rn$.
\\
From inequality \eqref{2} and the definition of Luxemburg  norm we deduce that
\begin{equation}\label{5}
\Big \| |\nabla U(z,t)| \, |t|^{1-s}\Big\|_{L^A (d\mu)}\leq c\, |u|_{s,A,\rn}.
\end{equation}
On the other hand,
\begin{align}\label{6}
\bigg \| \frac{\chi_{B_\delta (x,0)} (z,t) \, |t|^s}{(|x-z|^2 + t^2)^{\frac n2}}\bigg\|_{L^{\widetilde{A}}(d\mu)}
& =  \bigg \| \frac{\chi_{B_\delta (0,0)} (z,t) \, |t|^s}{(|z|^2 + t^2)^{\frac n2}}\bigg\|_{L^{\widetilde{A}}(d\mu)}
 \\ &
 = \inf\bigg \{ \lambda >0 : \int\int_{\{|z|^2+t^2 \leq \delta^2\}}
 \widetilde{A} \left(\frac{|t|^s}{\lambda \, (|z|^2+ t^2)^{\frac n2}}\right)\; d\mu \leq 1\bigg\}.
 \nonumber
\end{align}
Define the function $G \colon  (0, \infty) \to [0, \infty)$
as
$$G(t)= \int_t^\infty \frac{ \widetilde{A}(\tau)}{\tau^{1+ \frac n{n-s}}}\; d\tau\qquad \text{for $t >0$.}$$
One has that
\begin{align}\label{7}
 \int\int_{\{|z|^2+t^2 \leq \delta^2\}}
 \widetilde{A}& \left(\frac{|t|^s}{\lambda \, (|z|^2+ t^2)^{\frac n2}}\right) \; d\mu
 =
  \int\int_{\{t^2(|y|^2 +1) \leq \delta^2\}}
 \widetilde{A} \left(\frac{|t|^{s-n}}{\lambda \, (|y|^2+ 1)^{\frac n2}}\right)\, |t|^{n-1}\; dy \, dt
 \\ &
 = 2\, \int\int_{\{t^2(|y|^2 +1) \leq \delta^2, t>0\}}
 \widetilde{A} \left(\frac{t^{s-n}}{\lambda \, (|y|^2+ 1)^{\frac n2}}\right)\, t^{n-1}\; dy \, dt\nonumber
 \\&
 =
 \frac 2{(n-s)\, \lambda^{\frac n{n-s}}}\, \int_{\rn} \int_{\frac{\delta^{s-n}}{\lambda\, (|y|^2 + 1)^{\frac s2}}}^\infty \frac{ \widetilde{A}(\tau)}{\tau^{1+ \frac n{n-s}}}
\; d\tau\, \frac{dy}{(|y|^2 + 1)^{\frac n2 \frac{n}{n-s}}}\nonumber
\\&
=
\frac 2{(n-s)\, \lambda^{\frac n{n-s}}}\,\int_{\rn}  G\left(\frac{\delta^{s-n}}{\lambda\, (|y|^2 + 1)^{\frac s2}}\right) \, \frac{dy}{(|y|^2 + 1)^{\frac n2 \frac{n}{n-s}}}\nonumber
\\ &
=\frac{c}{\lambda^{\frac n{n-s}}} \int_0^\infty G\left(\frac{\delta^{s-n}}{\lambda\, (\tau^2 + 1)^{\frac s2}}\right) \, \frac{\tau^{n-1}}{(\tau^2 + 1)^{\frac n2 \frac{n}{n-s}}} \, d\tau\nonumber
\\ &
\leq \frac{c'}{\lambda^{\frac n{n-s}}} \int_0^1 G\left(\frac{\delta^{s-n}}{\lambda\, 2^{\frac s2}}\right) \, d\tau +
\frac{c}{\lambda^{\frac n{n-s}}} \int_1^\infty G\left(\frac{\delta^{s-n}}{\lambda\, (2\, \tau)^{s}}\right) \, \tau^{n-1 - n\frac{n}{n-s}}\, d\tau
\nonumber
\\& =
\frac {c'}{\lambda^{\frac n{n-s}}}\, G\left(\frac{\delta^{s-n}}{\lambda \, 2^{\frac s2}}\right) +
\frac{c}{\lambda^{\frac n{n-s}}} \int_1^\infty G\left(\frac{\delta^{s-n}}{\lambda\, (2\, \tau)^{s}}\right) \, \tau^{n-1 - n\frac{n}{n-s}}\, d\tau,
\nonumber
\end{align}
where  the first equality is due to the change of variables
$z= t\,y$, the fifth one to the use of polar coordinates, and the inequality
to the fact that $G$ is a decreasing function. Here, $c$ and $c'$ are constants depending on $n$ and $s$.
Moreover, by the change of variable $\theta= \frac{\delta^{s-n}}{\lambda \, (2\, \tau)^s}$,
\begin{align}\label{9}
\frac{1}{\lambda^{\frac n{n-s}}} \int_1^\infty G\left(\frac{\delta^{s-n}}{\lambda\, (2\, \tau)^{s}}\right) \, \tau^{n-1 - n\frac{n}{n-s}}\, d\tau= c'' \, \delta ^n \, \int_0^{\frac{\delta^{s-n}}{\lambda \, 2^s}} G(\theta) \, \theta^{\frac n{n-s} -1}\; d\theta
\end{align}
for some constant $c''=c''(n,s)$.
Notice that
\begin{equation}\label{10}
\EE(t)= t^{\frac n{n-s}}\, G(t) \qquad \text{for $t>0$,}
\end{equation}
where $\EE$ is the function defined by \eqref{\EE}.
Thereby,
\begin{equation}\label{10''}
 \int_0^t G(\theta)\, \theta^{\frac n{n-s} -1}\;d\theta = \int_0^t \frac{\EE(\theta)}{\theta} \; d\theta \leq \EE(t) \qquad \text{for $t>0$,}
\end{equation}
where the inequality holds owing to property \eqref{incr} applied to the Young function $\EE$.
Equations \eqref{9} and \eqref{10''} imply that
\begin{equation}\label{11}
\frac{1}{\lambda^{\frac n{n-s}}} \int_1^\infty G\left(\frac{\delta^{s-n}}{\lambda\, (2\, \tau)^{s}}\right) \, \tau^{n-1 - n\frac{n}{n-s}}\, d\tau \leq c''\, \delta ^n \, \EE\left(\frac{\delta^{s-n}}{\lambda \, 2^s}\right).
\end{equation}
From equations \eqref{7}, \eqref{10}, and \eqref{11} one infers that
\begin{align}\label{12}
\int\int_{\{|z|^2+ t^2 \leq \delta^2\}} \widetilde{A}\left(\frac{|t|^s}{\lambda(|z|^2 + t^2)^{\frac n2}}\right)\; d\mu\leq
 \frac{c'}{\lambda^{\frac n{n-s}}}\,
 G\left(\frac{\delta^{s-n}}{\lambda\, 2^{\frac s2}}\right)
 + c''\, \delta^n \EE\left( \frac{\delta^{s-n}}{\lambda\, 2^s}\right)
 \leq c \, \delta^n \EE\left( \frac{c\, \delta^{s-n}}{\lambda}\right)
\end{align}
for some constant $c=c(n,s)$.
Observe that the equation
$$c\, \delta^n\, \EE\left (\frac{c\, \delta^{s-n}}{\lambda}\right) =1$$
is equivalent to
$$\lambda = \frac c{\delta^{n-s} \EE^{-1} \left(\frac {\delta ^{-n}}{c}\right)}.
$$
Hence, by equations \eqref{min} and \eqref{AAtilde},
\begin{align}\label{july20}
    \lambda \approx \frac {1}{\delta^{n-s} \EE^{-1} \left(\delta ^{-n}\right)} \approx \vartheta_s(\delta) \qquad \text{for $\delta >0$,}
\end{align}
with equivalence constants depending on $n$ and $s$.
\\
From equations \eqref{6}, \eqref{12} and \eqref{july20}  we deduce that
\begin{equation}\label{13}
\bigg\|\frac{\chi_{B_{\delta}(x, 0)}(z,t)\, |t|^s}{\left(|x-z|^2 + t^2\right)^{\frac n2}}\bigg\|_{L^{\widetilde{A}}(d\mu)}
\leq c\, \vartheta_s (\delta)\qquad \hbox{for $x\in \rn$,}
\end{equation}
and for some constant $c=c(n,s)$, where $\vartheta_s$ is the function defined by \eqref{202}.
Inequalities \eqref{3}, \eqref{4}, \eqref{5} and \eqref{13} yield
\begin{equation}\label{14'}
    |u(x) -u(y)|\leq c\, \vartheta_s(|x-y|)\, |u|_{s,A,\rn}\qquad \hbox{for $x,y\in\rn$},
\end{equation}
and  for some constant $c=c(n,s)$. Inequality \eqref{23's<1} is thus established under the current assumption that $u$ is locally H\"older continuous.
\\ The next step consists of a derivation
of the modular inequality \eqref{paseky1}  from  \eqref{23's<1}.
To this purpose, set
$$M=  \int_{\rn} \int_{\rn} A\left(\frac{|u(x) - u(y)|}{|x-y|^s}\right)\; \frac{dx \, dy}{|x-y|^n},$$
and let $A_M$ be the Young function given by
\begin{equation}\label{AM}
A_M(t)= \frac{A(t)}M \qquad \text{ for $t \geq 0$.}
\end{equation}
Then, $$\widetilde{A_M}(t)= \frac 1M \widetilde A(Mt) \qquad \text{for $t \geq 0$,}$$
and the function $\EE_M$ defined as in \eqref{\EE}, with $A$ replaced by $A_M$, obeys the equation
$$\EE_M (t) = \frac 1M \EE(Mt) \qquad \text{for $t > 0$.}$$
Since $\EE_M^{-1}(t)= \frac 1M \EE^{-1}(Mt)$, the function $\vartheta_s^M$ obtained by replacing $A$ with $A_M$ in definition \eqref{202} is such that
\begin{equation}\label{sigmaM}
\vartheta_s^M(r) = \frac{M}{r^{n-s}\EE^{-1}(Mr^{-n})} \leq r^s \widetilde \EE^{-1}(Mr^{-n}) = r^s B^{-1}(Mr^{-n})\qquad \text{for $r>0$.}
\end{equation}
Notice that the first inequality holds owing to property \eqref{AAtilde}. Since, by the choice of $M$ and the definition of the seminorm,
$$|u|_{s,A_M, \rn}\leq 1,$$
an application of inequality \eqref{23's<1}  with $A$ replaced by $A_M$ yields \eqref{paseky1}.  This application is legitimate since the constant $c$ in \eqref{23's<1} is independent of $A$.
\\ Now, assume first that $A$ is finite-valued. Owing to  Theorem \ref{thm_modular}, there exist $\lambda_0>0$ and a sequence $\{u_k\}\subset  V^{s,A}(\rn)\cap C^\infty(\rn)$ such that
\begin{align}\label{july11}
    u_k(x) \to u(x) \qquad \text{for a.e. $x\in \rn$,}
\end{align}
and
\begin{align}\label{july6}
    \lim_{k\to \infty}J_{s,A}\Big(\frac{u_k-u}\lambda\Big)=0
\end{align}
for every $\lambda \geq \lambda_0$.
\\ Fix  $\varepsilon \in (0,1)$. The convexity of the function $A$ ensures that, for every $\lambda >0$,
\begin{align}\label{july4}
    J_{s,A}\Big(\frac{u_k}{\lambda}\Big)\leq \varepsilon J_{s,A}\Big(\frac{u_k-u}{\lambda\varepsilon}\Big) +(1-\varepsilon) J_{s,A}\Big(\frac{u}{\lambda (1-\varepsilon) }\Big)
\end{align}
for $k \in \N$.
By inequality \eqref{july4} and property \eqref{july6}, there exists a constant $c'$ such that
\begin{align}\label{july8}
    J_{s,A}\Big(\frac{u_k}{\lambda}\Big)\leq c'
\end{align}
for $k \in \mathbb N$, provided that $\lambda$ is sufficiently large.
An application of inequality \eqref{paseky1}
with $u$ replaced with $u_k$ tells us that
\begin{align}\label{july9}
    \frac{|u_k(x)-u_k(y)|}{\lambda}\leq c |x-y|^s B^{-1}\bigg(\frac 1{|x-y|^n}J_{s,A}\Big(\frac{u_k}{\lambda}\Big)\bigg)\qquad \text{for $x, y \in \rn$,}
\end{align}
and every $\lambda >0$. Combining inequalities \eqref{july8} and \eqref{july9}, and recalling equation \eqref{min} and definition \eqref{202} of the function $\vartheta_s$
yield
\begin{align}\label{july10}
\frac{|u_k(x)-u_k(y)|}{\lambda}\leq c\,\vartheta_s(|x-y|) \qquad \text{for $x, y \in \rn$,}
\end{align}
and for some constant $c$ independent of $k$.
This shows that the sequence $\{u_k\}$ is equicontinuous.
\\ Next, fix $x_0\in \rn$ for which the limit \eqref{july11} holds and $R>0$. Since
\begin{align*}
   \frac{|u_k(x)|}{\lambda}\leq \frac{|u_k(x)-u_k(x_0)|}{\lambda}+ \frac{|u_k(x_0)|}{\lambda} \qquad \text{for  $x\in \rn$,}
\end{align*}
and for $k \in \N$, from inequality \eqref{july10} we infer that the sequence $\{u_k\}$ is also equibounded in $B_R(x_0)$. Thanks to the arbitrariness of $R$, we deduce from Ascoli-Arzel\`a's theorem that there exists a subsequence of $\{u_k\}$, still indexed by $k$, and a continuous function $\overline u$ such that
\begin{align}\label{july12}
    u_k(x) \to \overline u(x)\qquad \text{for  $x\in \rn$.}
\end{align}
Equations \eqref{july11} and \eqref{july12} imply that
$$u= \overline u.$$
Owing to inequalities \eqref{july4} and \eqref{july9},
\begin{align}\label{july13}
     \frac{|u_k(x)-u_k(y)|}{\lambda}\leq c  |x-y|^s B^{-1}\bigg(\frac 1{|x-y|^n}
     \bigg(
     \varepsilon J_{s,A}\Big(\frac{u_k-u}{\lambda\varepsilon}\Big) +(1-\varepsilon) J_{s,A}\Big(\frac{u}{\lambda (1-\varepsilon) }\Big)\bigg)
     \bigg)
\end{align}
for some constant $c$, for $x, y \in \rn$, and for $k\in \mathbb N$. Passing to the limit as $k \to \infty$ in equation \eqref{july13} yields
\begin{align}\label{july14}
     \frac{|u(x)-u(y)|}{\lambda}\leq c  |x-y|^s B^{-1}\bigg(\frac {(1-\varepsilon)}{|x-y|^n}
      J_{s,A}\Big(\frac{u}{\lambda (1-\varepsilon) }\Big)
     \bigg) \qquad \text{for $x, y \in \rn$.}
\end{align}
Passing to the limit as $\varepsilon \to 0^+$ in the latter inequality, and replacing $u$ by $\frac{\lambda u}{|u|_{s,A,\rn}}$ in the resultant inequality yield  inequality \eqref{14'}.
This concludes the proof of inequality \eqref{23's<1} in the case when $A$ is finite-valued.

Assume next that $u \in V^{s,A}(\rn)$ and $A(t)=\infty$ for large $t$. Consider a  Young function $\overline {A}$ which agrees with $A$ near $0$ and equals  $t^p$ for large $t$, for some $p > \frac ns$. Since $A$ dominates $\overline A$ globally, one has that $ V^{s,A}(\rn)\to V^{s,\overline {A}}(\rn)$. By inequality \eqref{23's<1}  applied with $A$ replaced by $\overline A$,
\begin{equation}\label{july14bis}
   |u(x)-u(y)|\leq c \,\vartheta_s(|x-y|)\, |u|_{s, \overline A, \rn} \qquad \text{for $x, y\in \rn$,}
\end{equation}
and for some constant $c$.
One can verify that
$$\vartheta_s (r) \approx r^{s-\frac np} \qquad \text{for $r\leq 1$.}$$ Hence, we infer that $u$ is locally H\"older continuous. As a consequence, the above proof of inequality \eqref{14'} directly applies to the function $u$, and does not require any approximation argument.
\\ The proof of inequality \eqref{23's<1} is complete.

It remains to prove the necessity of condition \eqref{convinf} and the optimality of the modulus of continuity $\sigma_s$. To this purpose, assume that $\omega\colon  [0, \infty)\to [0, \infty)$ is a modulus of continuity such that
\begin{equation}\label{july16}
|u(x) - u(y)| \leq c\, \omega (|x-y|) \, |u|_{s,A,\rn} \qquad \text{for $x, y \in \rn$,}
\end{equation}
for some constant $c$, and for every  $u\in V^{s,A}(\rn)$.
Consider trial functions of the form
\begin{equation}\label{july15}
    u(x)= \int_{\omega_n |x|^n}^\infty f(r)\, r^{-1 + \frac sn}\,dr\qquad \hbox{for $x\in \rn$,}
\end{equation}
where the function $f\colon  [0, \infty) \to [0,\infty)$ is non-increasing, bounded, and with bounded support.
Thanks to ~\cite[Inequality (6.15)] {ACPS_emb}, there exists a constant $c=c(n,s)$ such that
\begin{equation}\label{july17}
    |u|_{s,A,\rn}\leq c\, \|f\|_{L^A(0, \infty)}.
\end{equation}
Inequalities \eqref{july16} and \eqref{july17} imply that there exist positive constants $c$ and $c'$ such that
\begin{align}
    \label{july18}
    c \geq \sup_{u\in V^{s,A}(\rn)}\, \sup_{x\in\rn} \frac{|u(x)-u(0)|}{\omega(|x|) \, |u|_{s,A,\rn}} \geq
        c'\sup_{f} \, \sup_{\rho>0}\frac{\int_0^{\omega_n\, \rho^n} f(r)\, r^{-1 +\frac sn} \; dr}{\omega (\rho)\, \|f\|_{L^A(0, \infty)}}
     =
\sup_{\rho>0}\, \sup_{f} \frac{\int_0^{\omega_n\, \rho^n} f(r)\, r^{-1 +\frac sn} \; dr}{\omega (\rho)\, \|f\|_{L^A(0, \infty)}},
    \end{align}
    where $f$ is as above. Any non-increasing function $f\colon  [0, \infty) \to [0,\infty)$ admits a pointwise monotone approximation by a sequence of non-increasing, bounded functions with bounded support. Hence, by the monotone convergence theorem for Lebesgue integrals and the Fatou property of Orlicz norms, the supremum on the rightmost side of equation \eqref{july18} can be extended to all nonnegative{\color{green!70!black!},}
    non-increasing functions $f$. Thanks to the Hardy-Littlewood inequality for rearrangements, the supremum in $f$
    on the rightmost side of equation \eqref{july18} does not increase if it is extended to all (non-necessarily non-increasing)  functions in $L^A(0, \infty)$. Therefore, by
     \eqref{revholder}, it agrees with
    $ \|r^{-1+\frac sn}\|_{L^{\widetilde{A}}(0, \omega_n\, \rho^n)}.$
    As a consequence,
    \begin{align}\label{july21}
        \sup_{\rho>0}\,  \frac{\|r^{-1+\frac sn}\|_{L^{\widetilde{A}}(0, \omega_n\, \rho^n)}}{\omega (\rho) }\leq c.
    \end{align}
On the other hand, since
$$ \|r^{-1+\frac sn}\|_{L^{\widetilde{A}(}0, \omega_n\, \rho^n)} = \inf\bigg\{ \lambda>0: \int_0^{\omega_n\, \rho^n} \widetilde A\Big(\frac{t^{-1+\frac sn}}\lambda\Big)\, dt\leq 1\bigg\},$$
computations analogous to \eqref{6}--\eqref{july20} show that the finiteness of this norm is equivalent to condition \eqref{convinf}, and that
\begin{equation}\label{july19}
    \|r^{-1+\frac sn}\|_{L^{\widetilde{A}}(0, \omega_n\, \rho^n)} \approx  \|r^{-1+\frac sn}\|_{L^{\widetilde{A}}(0, \rho^n)}\approx \vartheta_s(\rho) \qquad \text{for $\rho >0$},
\end{equation}
up to multiplicative constants depending on $n$ and $s$.
Inequalities \eqref{july18} and \eqref{july19} imply that
\begin{equation}\label{20}
   \vartheta_s(\rho)\leq c  \, \omega (\rho)\qquad\hbox{for $\rho>0$},
\end{equation}
 and for some constant $c$. Hence,
$$\mathcal{C}^{\sigma_s (\cdot)}(\rn) =\mathcal{C}^{\vartheta_s (\cdot)}(\rn)\to  \mathcal{C}^{\omega (\cdot)}(\rn).$$
\end{proof}

\subsection{Proof of Theorem \ref{thm:1<s<n}: case $s\in (1,n)$}

Dealing with embeddings of fractional-order Orlicz-Sobolev spaces of order $s>1$ calls for the use of a yet different method.  It requires a sharp combination of precise pieces of information derived from Theorems \ref{CP} and  \ref{a}. Moreover, because of the presence of two addends in the modulus of continuity defining the target space, the proof of its optimality
not only involves radial trial functions as in \eqref{july15}, but also functions of a different form, which depend radially in $n-1$ variables and are just odd in the remaining one. Properties of the relevant family of trial functions are established in the following lemma. The subsequent lemma is also needed to establish the sharpness of the modulus of continuity.

\begin{lemma}\label{lemma1}
Let  $s\in (1,  n+1)\setminus \N$.
Let $f$ be a nonnegative, non-increasing function in $L^A(0, \infty)$ with bounded support. Let $u$ be the function defined by
\begin{align}\label{sep108}
u(x) = x_1\, \int_{\omega_n |x|^n}^\infty \int_{r_1}^\infty \cdots \int_{r_{[s]}}^\infty f(r_{[s]+1})\,  r_{[s]+1} ^{-[s]-1+ \frac{s-1}{n}} \; d r_{[s]+1} \cdots  d r_1 \qquad \hbox{for}\;\; x\in \rn.
\end{align}
Then,
\begin{equation}\label{sep109}
|\nabla^{[s]}  u |_{\{s\}, A, \rn} \leq c \|f\|_{L^A(0, \infty)}\,
\end{equation}
for some constant $c=c(n,s)$.
\end{lemma}

\begin{proof}
Notice the alternative expression
$$
u(x) = \frac {1}{[s]!} \, x_1 \,\int_{\omega_n |x|^n}^\infty f(r)\,  r^{-[s]-1+ \frac{s-1}{n}} \, (r - \omega_n |x|^n)^{[s]} \; dr \qquad \hbox{for}\;\; x\in \rn,$$
which follows via Fubini's theorem and will be used in what follows without explicit mention.
\\
Inequality  \eqref{sep109} will follow if we show that
\begin{equation}\label{150}
\int_{\rn} \int_{\rn} A\left( \frac{ | \nabla ^{[s]} u(x) - \nabla^{[s]} u(y)|}{|x-y|^{\{s\}}}\right) \; \frac{dx \, dy}{|x-y|^n} \leq \int_0^\infty A\left ( c \, f(r)\right)\; dr
\end{equation}
for some constant  $c=c(n,s)$, where the fractional part $\{s\}$ of $s$ obeys:
\begin{align}\label{frac}
    \{s\}= s-[s].
\end{align}
One can verify that any ${[s]}-$th order partial derivative of $u$ is a linear combination of terms of the form
\begin{equation}\label{dec240}
x_{\alpha_1} \cdots x_{\alpha_i} \, |x|^{j n -{[s]} +1 -i} \, \int_{\omega_n |x|^n}^\infty \int_{r_{j+1}}^\infty \cdots
\int_{r_{[s]}}^\infty f(r_{{[s]}+1})\, r_{{[s]}+1}^{ -{[s]} -1 + \frac {s-1}n}\; d r_{{[s]}+1} \cdots  d r_{j+1}\quad \text{for a.e. $x \in \rn$,}
\end{equation}
where $i=0, 1, \dots , {[s]}+1$, $j=1, \dots , {[s]}$, and  $\alpha_i \in\{1, \dots , n\}$. Also, if ${[s]}=1$,  the values $i=j=0$ have to be included  as well.
\\
To estimate the left-hand side of inequality \eqref{150}, we observe that, by symmetry, we may restrict the  region of integration to the set $\{(x,y)\in \rn\times \rn: |x|\leq |y|\}$, and we split the latter into
 the subsets
 $\{|y| \geq 2 |x|\}$ and $\{2|x| > |y| \geq |x|\}$.
\\
\emph{Part 1}: Estimate over  $\{|y| \geq 2 |x|\}$.
Let us first assume that  ${[s]}>1$, and hence, $j \geq 1$ in  \eqref{dec240}. Trivially,
\begin{equation*}
|\nabla^{[s]} u(x) - \nabla^{[s]} u(y)| \leq |\nabla^{[s]} u(x)| + |\nabla^{[s]} u(y)| \qquad \text{for a.e. $x, y \in \rn$.}
\end{equation*}
We shall provide a bound for $|\nabla^{[s]} u(x)|$,  the bound for $|\nabla^{[s]} u(y)|$ being analogous, with $x$ replaced with $y$.
One has that
\begin{align*}
&|\nabla^{[s]} u(x)|  \lesssim \sum_{j=1}^{[s]} |x|^{j n -{[s]} +1} \int_{\omega_n |x|^n}^\infty \int_{r_{j+1}}^\infty \cdots \int_{r_{[s]}}^\infty f(r_{{[s]}+1})\,  r_{{[s]}+1} ^{-{[s]}-1+ \frac{s-1}n}
\; d r_{{[s]}+1} \cdots  d r_{j+1}
\\
 & \lesssim
\sum_{j=1}^{[s]} |x|^{j n -{[s]} +1} \int_{\omega_n |x|^n}^\infty f(r)\,  r^{-\textcolor{blue}{j}-1+ \frac{s-1}n} \; dr
  \lesssim \sum_{j=1}^{[s]} |x|^{j n -{[s]} +1} \, |x|^{-jn +s-1} \, f(\omega_n |x|^n)\lesssim |x|^{s-{[s]}} \, f(\omega_n |x|^n)\,, \nonumber
\end{align*}
for a.e. $x\in \rn$,
where the third inequality is due to the fact that  $f$ is non-increasing and $-j-1 + \frac{s-1}{[s]} < -1$ since $j\geq 1$. Here, and throughout this proof, the constant in the relation $\lq\lq \lesssim "$  depends on $n$ and $s$.
\\ Thus,
\begin{align*}
\int_{\Rn} &\int_{|y|\geq 2|x|} A\left(\frac{|\nabla^{[s]} u(x) - \nabla^{[s]} u(y)|}{|x-y|^{\{s\}}}\right) \frac{\,dx \,dy}{|x-y|^n}\\
&\lesssim
\int_{\Rn} \int_{|y| \geq 2|x|} A \left( c |x|^{\{s\}} |y|^{-\{s\}} f(\omega_n|x|^n) \right) \frac{dy}{|y|^n}\,dx
+ \int_{\Rn} \int_{|y| \geq 2|x|} A \left(c f(\omega_n|y|^n) \right) \,dx\frac{dy}{|y|^n}\\
&\lesssim \int_{\Rn} \int_{2|x|}^\infty A \left( c' |x|^{\{s\}} r^{-\{s\}} f(\omega_n|x|^n) \right) \frac{dr}{r}\,dx
+ \int_{\Rn} A \left( c' f(\omega_n|y|^n) \right) \,dy
\end{align*}
for some constants $c$ and $c'$ depending on $n$ and $s$. Notice that, in the last inequality, we have also made use of property \eqref{kt}.
The change of variables $t=c' |x|^{\{s\}} r^{-\{s\}} f(\omega_n|x|^n)$ yields
$$
\int_{2|x|}^\infty A \left( c' |x|^{\{s\}} r^{-\{s\}}  f(\omega_n|x|^n) \right) \frac{dr}{r}\,dx
\approx \int_0^{c' f(\omega_n |x|^n)} A(t) \frac{dt}{t}
\leq A\left(c' f(\omega_n |x|^n)\right),
$$
where the last inequality holds thanks to property \eqref{incr}.
Therefore,
\begin{align}\label{nov209}
\int_{\Rn} \int_{|y|\geq 2|x|} A\left(\frac{|\nabla^{[s]} u(x) - \nabla^{[s]} u(y)|}{|x-y|^{\{s\}}}\right) \frac{\,dx \,dy}{|x-y|^n}
\lesssim \int_{\Rn} A \left(c f(\omega_n|x|^n) \right) \,dx
= \int_0^\infty A \left(c f(r) \right) \,dr
\end{align}
for some constant $c=c(n,s)$.
\\
 Assume next that ${[s]}=1$. Besides the terms in \eqref{dec240} with $j=1$, which can be estimated as above,
we have now to estimate also the term with ${[s]}=1$, $i=j=0$. Altogether, we obtain that
\begin{align*}
\int_{\R^n}\int_{|y|\ge2|x|} &
	A\left(\frac{|\nabla u(x)-\nabla u(y)|}{|x-y|^{s-1}}\right)\frac{dx\,dy}{|x-y|^n}\\ &  \lesssim
\int_{\R^n}\int_{\{|y|\ge2|x|\}}
			A\left(\frac{c\int_{\omega_n|x|^n}^{\omega_n|y|^n}f(r)r^{-1+\frac{s-1}{n}}\,dr}{|y|^{s-1}}\right)\frac{dy}{|y|^n}\,dx+ \int_0^\infty A \left(c f(r) \right) \,dr
\end{align*}
for some constant $c=c(n,s)$.
\\
The following chain holds:
\begin{align}\label{nov200}
			\int_{\R^n}\int_{\{|y|\ge2|x|\}}&
			A\left(\frac{c\int_{\omega_n|x|^n}^{\omega_n|y|^n}f(r)r^{-1+\frac{s-1}{n}}\,dr}{|y|^{s-1}}\right)\frac{dy}{|y|^n}\,dx
				\\  \nonumber
		& \lesssim \int_{\R^n}\int_{\{|y|\ge2|x|\}}
			\int_{\omega_n|x|^n}^{\omega_n|y|^n}A\left(c'f(r)\right)r^{-1+\frac{s-1}{n}}\,dr\frac{dy}{|y|^{n+s-1}}\,dx
				\\  \nonumber
		&
\leq \int_{0}^{\infty}A\left(c'f(r)\right)r^{-1+\frac{s-1}{n}}
				\int_{\{\omega_n|x|^n<r\}}\int_{\{|y|\ge2|x|\}}\frac{dy}{|y|^{n+s-1}}\,dx\,dr
				\\  \nonumber
		& \lesssim \int_{0}^{\infty}A\left(c'f(r)\right)r^{-1+\frac{s-1}{n}}
				\int_{\{\omega_n|x|^n<r\}}|x|^{-s+1}\,dx\,dr
				\\  \nonumber
		& \lesssim \int_{0}^{\infty}A\left(c'f(r)\right)r^{-1+\frac{s-1}{n}}r^{1-\frac{s-1}{n}}
				\,dr
				 \approx \int_{0}^{\infty}A\left(c'f(r)\right)\,dr
\end{align}
for some positive constants $c$ and $c'$ depending on $n$ and $s$.
In particular, the first inequality relies upon Jensen's inequality. Hence, inequality \eqref{nov209} holds also for ${[s]}=1$.

\medskip

\noindent
\emph{Part 2}: Estimate over $\{2|x| \geq |y| \geq |x|\}$.
First, assume that ${[s]}>1$ and hence $j\geq 1$ in \eqref{dec240}. Thereby,
\begin{align*}
&|\nabla^{[s]} u(x) - \nabla^{[s]} u(y)|
\cr &
\lesssim \sum_{i=0}^{{[s]}+1} \sum_{(\alpha_1, \dots, \alpha_i)} \sum_{j=1}^{[s]} \bigg |x_{\alpha_1} \cdots  x_{\alpha_i} \, |x|^{jn -{[s]} +1-i} \, \int_{\omega_n |x|^n}^\infty \int_{r_{j+1}}^\infty \cdots
\int_{r_{[s]}}^\infty f(r_{{[s]}+1})\, r_{{[s]}+1}^{ -{[s]} -1 + \frac {s-1}n}\; d r_{{[s]}+1} \cdots  d r_{j+1} \nonumber
\cr& \quad \quad \quad \quad \quad \quad\quad \quad \quad
- y_{\alpha_1} \cdots y_{\alpha_i} \, |y|^{jn -{[s]} +1-i} \, \int_{\omega_n |y|^n}^\infty \int_{r_{j+1}}^\infty \cdots
\int_{r_{[s]}}^\infty f(r_{{[s]}+1})\, r_{{[s]}+1}^{ -{[s]} -1 + \frac {s-1}n}\; d r_{{[s]}+1} \cdots  d r_{j+1}\bigg | \nonumber
\cr &
\lesssim \sum_{i=0}^{{[s]}+1} \sum_{(\alpha_1, \dots, \alpha_i)} \sum_{j=1}^{[s]} \bigg |x_{\alpha_1} \cdots  x_{\alpha_i} \, |x|^{jn -{[s]} +1-i} - y_{\alpha_1} \cdots y_{\alpha_i} \, |y|^{jn -{[s]} +1-i}\bigg| \nonumber
\cr &  \quad \quad \quad \quad \quad \quad\quad \quad \quad
\times \,
\int_{\omega_n |y|^n}^\infty \int_{r_{j+1}}^\infty \cdots
\int_{r_{[s]}}^\infty f(r_{{[s]}+1})\, r_{{[s]}+1}^{ -{[s]} -1 + \frac {s-1}n}\; d r_{{[s]}+1} \cdots  d r_{j+1}\nonumber
\cr &
+  \sum_{i=0}^{{[s]}+1} \sum_{(\alpha_1, \dots, \alpha_i)} \sum_{j=1}^{[s]} \bigg |x_{\alpha_1} \cdots  x_{\alpha_i} \, |x|^{jn -{[s]} +1-i}\, \int_{\omega_n |x|^n}^{\omega_n |y|^n} \int_{r_{j+1}}^\infty \cdots
\int_{r_{[s]}}^\infty f(r_{{[s]}+1})\, r_{{[s]}+1}^{ -{[s]} -1 + \frac {s-1}n}\; d r_{{[s]}+1} \cdots  d r_{j+1}\bigg| \nonumber
\cr &
 = I + II
\end{align*}
for a.e. $x, y$ such that $2|x| \geq |y| \geq |x|$.
By \cite[Lemma 7.8]{ACPS_emb} (which still holds if the assumption $i\leq n$ appearing  in the statement is dropped), we deduce that, for the same $x$ and $y$,
\begin{align*}
I &\lesssim \sum_{j=1}^{[s]} |x-y|\, |x|^{jn-{[s]}}\; \int_{\omega_n |y|^n}^\infty \int_{r_{j+1}}^\infty \cdots
\int_{r_{[s]}}^\infty f(r_{{[s]}+1})\, r_{{[s]}+1}^{ -{[s]} -1 + \frac {s-1}n}\; d r_{{[s]}+1} \cdots  d r_{j+1}
\cr &
\lesssim \sum_{j=1}^{[s]} |x-y|\, |x|^{jn-{[s]}}\; \int_{\omega_n |y|^n}^\infty f(r)\, r^{ -j-1 + \frac {s-1}n}\; d r
\leq \sum_{j=1}^{[s]} |x-y|\, |x|^{jn-{[s]}}\; \int_{\omega_n |x|^n}^\infty f(r)\, r^{ -j -1 + \frac {s-1}n}\; d r \nonumber
\cr &
\lesssim \sum_{j=1}^{[s]} |x-y|\, |x|^{jn-{[s]}}\;  f(\omega_n |x|^n)\, |x|^{ -jn  + s-1}
\lesssim  |x|^{\{s\} -1} |x-y|\,  f(\omega_n |x|^n), \nonumber
\end{align*}
where the third inequality holds since $|y|\geq |x|$, and the fourth since $-j-1+ \frac {s-1}n < -1$.
\\
On the other hand,
\begin{align*}
II &\lesssim \sum_{j=1}^{{[s]}-1} |x|^{jn -{[s]} +1} \, \int_{\omega_n |x|^n}^{\omega_n |y|^n} \int_r^{\infty} f(\varrho) \,  \varrho^{-j-2+\frac{s-1}n} \, d\varrho\, dr + |x|^{{[s]n}- {[s]} +1}
\int_{\omega_n |x|^n}^{\omega_n |y|^n}f(r) \,  r^{-{[s]}-1+\frac{s-1}n} \,  dr
\\
\nonumber &
\lesssim \sum_{j=1}^{{[s]}-1} |x|^{jn -{[s]} +1} \,f(\omega_n |x|^n)\, |x|^{-jn -n +s -1}(|y|^n -|x|^n) + |x|^{{[s]n} -{[s]} +1}\,f(\omega_n |x|^n)\,|x|^{-mn-n+s-1} (|y|^n -|x|^n)
\\
\nonumber&
\lesssim  |x|^{\{s\} -n}\,f(\omega_n |x|^n)\,|x|^{n-1}\, |y-x| \, + |x|^{\{s\} -n }\,f(\omega_n |x|^n)
|x|^{n-1}\, |y-x|
\cr &
\lesssim
|x|^{\{s\} -1}\,|y-x| \,f(\omega_n |x|^n)
\end{align*}
for a.e. $x, y$ such that $2|x| \geq |y| \geq |x|$.
If ${[s]}=1$, then the terms in the expression $\nabla u(x) - \nabla u(y)$ corresponding to $j \geq 1$ can be estimated as above, whereas
 the term   corresponding to  $i=j=0$ in \eqref{dec240} can be estimated, in absolute value, by
\begin{align*}
\int_{\omega_n |x|^n}^{\omega_n |y|^n} \int_r^{\infty} f(\varrho) \,  \varrho^{-2+\frac{s-1}n} \, d\varrho\, dr &
\lesssim
f(\omega_n |x|^n)\, \int_{\omega_n |x|^n}^{\omega_n |y|^n}
r^{-1 +\frac{s-1}n}\;dr
\cr &
\lesssim
f(\omega_n |x|^n)\, |x|^{-n+s-1}(|y|^{n}- |x|^{n})
\lesssim
|y-x|\,|x|^{s-2} \, f(\omega_n |x|^n)\,.
\end{align*}
Therefore, for every ${[s]}\in \{1,\dots, n\}$,
\begin{align*}
\int_{\Rn} &\int_{|x| \leq |y| \leq 2|x|} A\left(\frac{|\nabla^{[s]} u(x) - \nabla^{[s]} u(y)|}{|x-y|^{\{s\}}}\right) \frac{\,dx \,dy}{|x-y|^n}\\
&\lesssim  \int_{\Rn} \int_{|x| \leq |y| \leq 2|x|} A\left(c|x|^{\{s\}-1}|y-x|^{ 1-\{s\}}f(\omega_n |x|^n)\right)\frac{dy}{|y-x|^n}\,dx\\
&= \int_{\Rn} \int_{|x| \leq |z+x| \leq 2|x|} A\left(c|x|^{\{s\}-1}|z|^{1-\{s\}}f(\omega_n |x|^n)\right)\frac{dz}{|z|^n}\,dx\\
&\leq\int_{\Rn} \int_{|z|\leq 3|x|} A\left(c|x|^{\{s\}-1}|z|^{1-\{s\}}f(\omega_n |x|^n)\right)\frac{dz}{|z|^n}\,dx\\
&\lesssim \int_{\Rn} \int_{0}^{3|x|} A\left(c|x|^{\{s\}-1} r^{1-\{s\}}f(\omega_n |x|^n)\right)\frac{dr}{r}\,dx
\end{align*}
for some positive constant $c=c(n,s)$.
The change of variables $t=c|x|^{\{s\}-1} r^{1-\{s\}}f(\omega_n |x|^n)$ yields
$$
\int_{0}^{3|x|} A\left(c|x|^{\{s\}-1}r^{ 1-\{s\}}f(\omega_n |x|^n)\right)\frac{dr}{r}
\approx \int_0^{c f(\omega_n |x|^n)} A(t) \frac{\,dt}{t}
\leq A(c f(\omega_n |x|^n))
$$
for $x\in \rn$.
Thereby,
\begin{align}\label{nov210}
\int_{\Rn} \int_{|x|\leq|y|\leq 2|x|} A\left(\frac{|\nabla^{[s]} u(x) - \nabla^{[s]} u(y)|}{ |x-y|^{\{s\}}}\right) \frac{\,dx \,dy}{|x-y|^n}
\lesssim \int_{\Rn} A \left( c f(\omega_n|x|^n) \right) \,dx
= \int_0^\infty A \left(c f(r) \right) \,dr.
\end{align}
Inequality \eqref{150} follows from inequalities \eqref{nov209} and \eqref{nov210}.
\end{proof}

\medskip

\begin{lemma}\label{lemma2}
Let $s>1$ and let $A$ be a Young function and $R>0$. Let $g$ be a nonnegative function in $L^A(0, \infty)$, vanishing on $(0, R)$, and
 non-increasing in $(R, \infty)$. Then, there exists a  nonnegative  non-increasing function $f\in L^A(0, \infty)$ such that
\begin{equation*}
\|f\|_{L^A(0, \infty)}\leq c \|g\|_{L^A(0, \infty)}
\end{equation*}
for some  absolute constant $c$, and
\begin{equation*}
\int_R^\infty g(r) \, r^{-1 - \frac{s-1}n} \; dr \approx \int_R^\infty f(r) \, r^{-1 - \frac{s-1}n} \; dr\,,
\end{equation*}
up to multiplicative constants depending on $n$ and $s$.
\end{lemma}

\begin{proof}
Consider the function $g_R\colon  (0, \infty) \to [0, \infty)$ defined by
\begin{equation*}
g_R(r) =
\begin{cases}
 \frac 1R \,\int_R^{2R} g(\rho) \; d\rho &\quad \hbox{if $r\in (R, 2R)$}
\cr
g(r) &\quad \hbox{otherwise.}
\end{cases}
\end{equation*}
One can verify, via Jensen's inequality, that
\begin{equation*}
\|g_R\|_{L^A(0, \infty)}\leq \|g\|_{L^A(0, \infty)}.
\end{equation*}
Moreover,
\begin{equation*}
\int_R^{2R} g(r) \, r^{-1 - \frac{s-1}n} \; dr \approx \int_R^{2R} g_R(r) \, r^{-1 - \frac{s-1}n} \; dr,
\end{equation*}
with equivalence constants depending on $n$ and $s$.
Consequently,
\begin{equation*}
\int_R^\infty g(r) \, r^{-1 - \frac{s-1}n} \; dr \approx \int_R^\infty g_R(r) \, r^{-1 - \frac{s-1}n} \; dr\,.
\end{equation*}
Now, define the function $f\colon  (0, \infty) \to [0, \infty)$ as
\begin{equation*}
f(r) =
\begin{cases}
\frac 1R \int_R^{2R} g(\rho) \; d\rho &\quad \hbox{if $r\in(0, 2R)$}
\cr
g(r) &\quad \hbox{if $r\in [2R, \infty)$.}
\end{cases}
\end{equation*}
Thanks to the monotonicity of $g$, the function $f$ is non-increasing on $(0, \infty)$.
\\
Since $f= g_R$ in $(R, \infty)$,  we have that
\begin{equation*}
\int_R^\infty f(r) \, r^{-1 - \frac{s-1}n} \; dr \approx \int_R^\infty g(r) \, r^{-1 - \frac{s-1}n} \; dr,
\end{equation*}
with equivalence constants depending on $n$ and $s$.
Furthermore, $f^* (r) \leq g^*_R\left(\frac r2 \right)$ for $r >0$. Hence, thanks to the boundedness of the dilation operator on $L^A(0,\infty)$,  with a norm bounded by an absolute constant, one has that
\begin{equation*}
\|f\|_{L^A(0, \infty)} \leq c \|g_R\|_{L^A(0, \infty)} \leq c
\|g\|_{L^A(0, \infty)}
\end{equation*}
for some absolute constant $c$.
\end{proof}

We are now in a position to accomplish the proof of Theorem \ref{thm:1<s<n}.

\begin{proof}[Proof of Theorem \ref{thm:1<s<n}]
Our approach to
 embedding \eqref{231<s<n} relies upon a combination of an embedding for fractional  Orlicz-Sobolev spaces of order $s-1$ with an embedding of first-order Sobolev type spaces into spaces of uniformly continuous functions.
 \\ Unless otherwise stated, the constants involved in this proof depend on $n$, $s$ and $A$.
 \\
 To begin with, observe that, if $u\in V^{s,A}_{d,1}(\rn)$, then  $\nabla u \in V^{s-1, A}_{d}(\rn)$, and
\begin{equation}\label{29}
    \big|\nabla^{[s]-1} (\nabla u)\big|_{\{s\}, A, \rn} = \big|\nabla ^{[s]} u\big|_{\{s\}, A, \rn}.
\end{equation}
First, assume that condition \eqref{july30} is fulfilled.
Then, Theorem \ref{a} yields
\begin{align}\label{30}
   V^{s-1, A}_{d}(\rn) \to
      L(\widehat{A}_{\frac n{s-1}}, \tfrac n{s-1})(\rn).
\end{align}
Equations \eqref{29} and \eqref{30} tell us that
\begin{equation}\label{31}
 \big\|\nabla u\big\|_{ L\big(\widehat{A}_{\frac n{s-1}}, \tfrac n{s-1}\big)(\rn)} \leq c \, \big |\nabla^{[s]} u\big|_{\{s\}, A, \rn}
\end{equation}
for some constant $c$ and for $u\in V^{s,A}_{d,1}(\rn)$.
By Theorem  \ref{CP},
\begin{equation}\label{32}
   V^1   L\big(\widehat{A}_{\frac n{s-1}}, \tfrac n{s-1}\big)(\rn) \to \mathcal{C}^{\nu_s (\cdot)} (\rn),
\end{equation}
namely
\begin{equation}\label{32norm}
   |u|_{\mathcal{C}^{\nu_s (\cdot)} (\rn)}\leq c \, \|\nabla u\|_{L\big(\widehat{A}_{\frac n{s-1}}, \tfrac n{s-1}\big)(\rn)}
\end{equation}
for some constant $c$ and for $u\in  V^1   L\big(\widehat{A}_{\frac n{s-1}}, \tfrac n{s-1}\big)(\rn)$,
where $\nu_s \colon  [0, \infty) \to [0, \infty)$ is the function given by
\begin{equation}\label{33}
   \nu_s(r)= \Big \| \rho^{-1+ \frac 1n}\, \chi_{(0, r^n)}(\rho) \Big\|_{L\big(\widehat{A}_{\frac n{s-1}}, \frac n{s-1}\big)'(\rn) } \qquad \text{for $r>0$}.
\end{equation}
Altogether,
\begin{equation}\label{july50}
   |u|_{\mathcal{C}^{\nu_s (\cdot)} (\rn)}\leq c  \, \big |\nabla^{[s]} u\big|_{\{s\}, A, \rn}
   \end{equation}
for some constant $c$ and for $u\in V^{s,A}_{d,1}(\rn)$.
\\
By  \cite[Lemma 4.5]{cianchi_MZ},
\begin{align}\label{34}
   &\Big \| \rho^{-1+ \frac 1n}\, \chi_{(0, r^n)}(\rho) \Big\|_{L\big(\widehat{A}_{\frac n{s-1}}, \frac n{s-1}\big)'(0, \infty) } \approx
    \Big \| \rho^{ \frac {s-1}n}\, \left((\cdot)^{-1+ \frac 1n}\, \chi_{(0, r^n)}(\cdot) \right)^{**} (\rho) \Big\|_{L^{\widetilde{A}}(0, \infty)}
    \qquad \text{for $r> 0$,}
    \end{align}
    with equivalence constants depending on $n$ and $s$. Moreover, one can verify that
    \begin{align}\label{july57}
     \Big \|& \rho^{ \frac {s-1}n}\, \left((\cdot)^{-1+ \frac 1n}\, \chi_{(0, r^n)}(\cdot) \right)^{**} (\rho) \Big\|_{L^{\widetilde{A}}(0, \infty)} \\ &
    \approx \Big \| \rho^{-1+ \frac sn}\, \chi_{(0, r^n)}(\rho) \Big\|_{L^{\widetilde{A}}(0, \infty) }  + r\, \Big \| \rho^{-1+ \frac {s-1}n}\, \chi_{(r^n, \infty)}(\rho) \Big\|_{L^{\widetilde{A}}(0, \infty) }\nonumber
      \qquad \text{for $r> 0$,}
\end{align}
with equivalence constants depending on $n$ and $s$.
 Computations show that
 \begin{align}\label{july34}
     r\, \Big \| \rho^{-1+ \frac {s-1}n}\, \chi_{(r^n, \infty)}(\rho) \Big\|_{L^{\widetilde{A}}(0, \infty) } \approx \varrho_s(r) \qquad \text{for $r>0$,}
 \end{align}
 up to multiplicative constants depending on $n$ and $s$, see e.g. \cite[Equation (6.24)]{CiRa}.
 Hence, by equation \eqref{july19},
 \begin{equation}\label{sep10}
     \nu_s(r) \approx \vartheta_s(r) + \varrho_s(r) \qquad \text{for $r>0$.}
 \end{equation}
 Combining equations \eqref{july50} and \eqref{sep10}  yields inequality \eqref{23'1<s<n}.
 \\ Next, assume that
 condition \eqref{july30} fails, namely \eqref{24} is in force. Under this assumption, Theorem \ref{a} tells us that
\begin{equation}\label{july31}
  V^{s-1, A}_{d}(\rn) \to
      \big(L^\infty \cap L(\widehat{A}_{\frac n{s-1}}, \tfrac n{s-1})\big)(\rn).
\end{equation}
 Embedding \eqref{july31} amounts to the inequality
 \begin{equation}\label{july54}
   \big \|\nabla u\big\|_{ \big(L^\infty \cap L(\widehat{A}_{\frac n{s-1}}, \tfrac n{s-1})\big)(\rn)}  \leq c\, \big|\nabla ^{[s]} u\big|_{\{s\}, A, \rn}
 \end{equation}
 for some constant $c$ and for  $u\in V^{s,A}_{d,1}(\rn)$.
 By Theorem \ref{CP} again,
 \begin{equation}\label{36}
  V^1   \big(L^\infty \cap L(\widehat{A}_{\frac n{s-1}}, \tfrac n{s-1}\big)(\rn) \to \mathcal{C}^{\mu_s (\cdot)} (\rn),
 \end{equation}
 where
 \begin{equation}\label{37}
 \mu_s(r)=\Big \| \rho^{-1+ \frac 1n}\, \chi_{(0, r^n)}(\rho) \Big\|_{\big(L^\infty \cap L(\widehat{A}_{\frac n{s-1}}, \frac n{s-1})\big)'(0, \infty) } \qquad \text{for $r >0$.}
 \end{equation}
 Thus,  there exists a constant $c$ such that
\begin{equation}\label{july58}
   |u|_{\mathcal{C}^{\mu_s (\cdot)} (\rn)}\leq c \, \|\nabla u\|_{\big(L^\infty \cap L(\widehat{A}_{\frac n{s-1}}, \tfrac n{s-1})\big)(\rn)}
\end{equation}
for $u \in   \big(L^\infty \cap L(\widehat{A}_{\frac n{s-1}}, \tfrac n{s-1})\big)(\rn)$.
By formula \eqref{associate-of-cap},
\begin{equation}\label{38}
   \big(L^\infty \cap L(\widehat{A}_{\frac n{s-1}}, \tfrac n{s-1})\big)'(0, \infty) = \big(L^\infty\big)' (0, \infty) +L\big(\widehat{A}_{\frac n{s-1}}, \tfrac n{s-1}\big)'(0, \infty),
\end{equation}
up to equivalent norms.
Thus, by equation \eqref{33},
\begin{align}\label{39}
   \mu_s(r ) & =\Big \| \rho^{-1+ \frac 1n}\, \chi_{(0, r^n)}(\rho) \Big\|_{\big(L^\infty \cap L(\widehat{A}_{\frac n{s-1}}, \frac n{s-1})\big)'(0, \infty) }
  \\ &
  = \inf_{f+g=\rho^{-1+ \frac 1n}\, \chi_{(0, r^n)}(\rho)} \left(\|f\|_{L^1(0, \infty)} + \|g\|_{L\big(\widehat{A}_{\frac n{s-1}}, \frac n{s-1}\big)'(0, \infty)}\right)\nonumber
    \\ &
    \leq \min\Big\{\Big\| \rho^{-1+ \frac 1n}\, \chi_{(0, r^n)}(\rho) \Big\|_{L^1(0, \infty)} , \Big\| \rho^{-1+ \frac 1n}\, \chi_{(0, r^n)}(\rho) \Big\|_{L\big(\widehat{A}_{\frac n{s-1}}, \frac n{s-1}\big)'(0, \infty)}\Big\} \nonumber
    \\ &
    \approx \min\{r, \nu_s(r)\}\nonumber
\end{align}
for $r > 0$.
We claim that
\begin{equation}\label{40}
    \min \{r, \nu_s(r)\} \approx
    \begin{cases}
r \quad & \hbox{if} \;\; 0<r < 1
\\
\nu_s(r) \quad &   \hbox{if}\;\;  r \geq 1 .
    \end{cases}
\end{equation}
 Owing to equation \eqref{sep10}, the inequality
\begin{equation}\label{41}
    \nu_s(r) \lesssim r \qquad \hbox{for}\;\; r\geq 1
\end{equation}
will follow if we show that
\begin{equation}\label{sep11}
    \liminf_{r\to \infty} \frac r{\vartheta_s(r)} >0
\end{equation}
and
\begin{equation}\label{sep12}
    \liminf_{r\to \infty} \frac r{\varrho_s(r)} >0.
\end{equation}
By the second equivalence in \eqref{july20}, condition \eqref{sep11} amounts  to
\begin{equation}\label{42}
  \liminf_{t\to \infty} t^{n-(s-1) }\, \EE^{-1} (t^{-n} )>0.
  \end{equation}
Analogously, condition \eqref{sep12} is equivalent to
\begin{equation}\label{42'}
  \liminf_{t\to \infty} t^{n-(s-1) }\, \FF^{-1} (t^{-n}) >0 .
  \end{equation}
We have that
 \begin{align*}
  &\liminf_{t\to \infty} t^{n-(s-1) }\, \EE^{-1} (t^{-n} )
  >0 \qquad\text{if and only if}\qquad
  \liminf_{t\to 0^+} \frac {t^{\frac {n}{n-(s-1)}}}{\EE(t)} >0.
  \end{align*}
 The latter condition reads
 \begin{equation}\label{45}
    \liminf_{t\to 0^+} \frac 1{t^{\frac n{n-s} - \frac n{n-(s-1)}}\,  \int_t^\infty \frac{\widetilde{A}(\tau)}{\tau^{1+\frac n{n-s}}}\; d\tau}>0.
    \end{equation}
Inequality \eqref{45} actually holds,
since $\frac {n}{n-s} - \frac{n}{n-(s-1)} >0$, whence
\begin{align}\label{46}
   & \limsup_{t \to 0^+} t^{\frac {n}{n-s} - \frac{n}{n-(s-1)}} \int_t^\infty  \frac{\widetilde{A}(\tau)}{\tau^{1+\frac n{n-s}}}\; d\tau
    \\ & \leq
    \lim_{t \to 0^+}  \int_t^\infty
    \tau^{\frac {n}{n-s} - \frac{n}{n-(s-1)}} \frac{\widetilde{A}(\tau)}{\tau^{1+\frac n{n-s}}}\; d\tau
    = \int_0^\infty \frac{\widetilde{A}(\tau)}{\tau^{1+\frac n{n-(s-1)}}}\; d\tau < \infty , \nonumber
\end{align}
where the last inequality is due to equations \eqref{conv0'} and \eqref{ibero0}. (In fact, one can show that the limit in \eqref{46} equals $0$, whence the limit in \eqref{45} is $\infty$.) Hence, \eqref{42} is established.
\\Similarly, property \eqref{42'} is equivalent to
  \begin{equation}\label{47}
    \liminf_{t\to 0^+} \frac 1{\int_0^t \frac {\widetilde{A}(\tau)}{\tau^{1+\frac n{n-(s-1)}}}\; d\tau} >0,
    \end{equation}
an inequality which holds, since the limit trivially equals infinity.
\\
On the other hand, to show that
\begin{equation}\label{49}
    \nu_s(r) \gtrsim r \qquad \hbox{if}\;\; 0<r \leq 1
\end{equation}
it suffices to verify that
\begin{equation}\label{sep12'}
    \liminf_{r\to  0^+} \frac {\varrho_s(r)}r >0.
\end{equation}
The latter inequality
is equivalent to
\begin{equation}\label{50}
    \limsup_{t\to 0^+} t^{n-(s-1)}\FF^{-1}(t^{-n}) <\infty,
\end{equation}
and, hence, to
\begin{equation}\label{52}
    \liminf_{t \to \infty} \frac{\FF(t)}{t^{\frac n{n-(s-1)}}} >0.
\end{equation}
This inequality holds since the $\liminf$ in \eqref{52} is, in fact, a limit and agrees with
$$
    \int_0^\infty \frac{\widetilde{A} (t)}{t^{1+\frac {n}{n-(s-1)}}} \; dt.
$$
Inequality \eqref{49} is therefore established.
Equation \eqref{40} is fully proved.
\\
This concludes the proof of embedding \eqref{231<s<n}.
\\ As for the fact that the function $\sigma_s$ is a modulus of continuity, property (ii) of Propositions \ref{prop:E} and \ref{prop:F} ensure that $\lim_{r \to 0^+}\sigma_s(r)=0$. Furthermore, equations \eqref{july19},  \eqref{july57} and \eqref{july34} tell us that
\begin{equation}\label{dec210}
\vartheta _s(r) + \varrho _s(r) \approx  \Big \| \rho^{-1+ \frac sn}\, \chi_{(0, r^n)}(\rho) \Big\|_{L^{\widetilde{A}}(0, \infty) }  + r\, \Big \| \rho^{-1+ \frac {s-1}n}\, \chi_{(r^n, \infty)}(\rho) \Big\|_{L^{\widetilde{A}}(0, \infty) }
      \qquad \text{for $r>0$.}
      \end{equation}
 The right-hand side of equation \eqref{dec210} is equivalent to
      \begin{equation}\label{dec211}
      \Big \| \rho^{-1+ \frac sn}\, \chi_{(0, r^n)}(\rho)+ r\,  \rho^{-1+ \frac {s-1}n}\, \chi_{(r^n, \infty)}(\rho) \Big\|_{L^{\widetilde{A}}(0, \infty) }    \qquad \text{for $r> 0$.}
      \end{equation}
The norm in \eqref{dec211} is non-decreasing in $r$, since the function
$$(0,\infty) \ni r \mapsto \rho^{-1+ \frac sn}\, \chi_{(0, r^n)}(\rho)+ r\,  \rho^{-1+ \frac {s-1}n}\, \chi_{(r^n, \infty)}(\rho)$$
is non-decreasing in $r$ for each $\rho \in (0, \infty)$. Altogether,  $\vartheta _s + \varrho _s$ is equivalent to a non-decreasing function, whence $\sigma_s$ is also equivalent to a non-decreasing function.

As far as the optimality of $\sigma_s$ is concerned, assume that $\omega$ is a modulus of continuity such that
\begin{equation}\label{53}
  V^{s, A}_{d,1}(\rn)\to \mathcal{C}^{\omega(\cdot)} (\rn).
\end{equation}
Namely, there exists a constant $c$ such that
\begin{equation}\label{54}
    |u(x)- u(y)| \leq c\, \omega \left(|x-y|\right) \, \big|\nabla ^{[s]} u \big|_{\{s\}, A, \rn}
\end{equation}
 for $u\in  V^{s, A}_{d,1}(\rn)$.
\\
Assume first that condition \eqref{july30} holds.
Given any nonnegative, non-increasing function $f\in L^A(0, \infty)$ with  support contained in $[0,L]$ for some fixed $L$, define the function
$u \colon \rn \to [0, \infty)$ as
\begin{equation}\label{fi1}
    u(x)= \int_{\omega_n |x|^n}^\infty \int_{r_1}^\infty \cdots \int_{r_{[s]}}^\infty f(r_{[s] +1})\, r_{[s]+1}^{-[s] -1 + \frac sn}
\; dr_{[s]+1}\cdots dr_1\qquad \hbox{for $x\in \rn$}.
\end{equation}
One can verify that $u$ is $[s]$-times weakly differentiable, and that $|\{|\nabla u| ^k >t\}|<\infty$ for $k=0, 1, \dots, [s]$ and  every $t>0$. Owing to Fubini's Theorem, it follows that
\begin{equation}\label{fi2}
    u(x)= \frac 1{[s]!}\int_{\omega_n |x|^n}^\infty  f(r)\, r^{-[s] -1 + \frac sn} (r - \omega_n |x|^n)^{[s]}\; dr
    \qquad \hbox{for $x\in \rn$}.
\end{equation}
Since
\begin{align*}
    u(0) & =\frac 1{[s]!}
  \int_0^\infty  f(r)\, r^{ -1 + \frac sn}\; dr \geq
\frac 1{[s]!}\int_{\omega_n |x|^n}^\infty  f(r)\, r^{ -1 + \frac sn}\; dr
\\ & \geq \frac 1{[s]!}\int_{\omega_n |x|^n}^\infty  f(r)\, r^{-[s] -1 + \frac sn} (r - \omega_n |x|^n)^{[s]}\; dr =u(x) \qquad \hbox{for $x\in \rn$},
\end{align*}
we have that
\begin{align}\label{fi3}
   \big| u(0) - u(x)\big|& \geq \frac 1{[s]!}\bigg(
  \int_0^\infty  f(r)\, r^{ -1 + \frac sn}\; dr-
  \int_{\omega_n |x|^n}^\infty  f(r)\, r^{ -1 + \frac sn}\; dr\bigg )
   \\ \nonumber &=  \frac 1{[s]!}\int_0^{\omega_n |x|^n}  f(r)\, r^{ -1 + \frac sn}\; dr \qquad  \hbox{for $x\in \rn$.}
\end{align}
Thanks to \cite[Lemma 7.6]{ACPS_emb},
\begin{align}\label{sep105}
\big|\nabla ^{[s]} u \big|_{\{s\}, A, \rn}\lesssim \|f\|_{L^A(0,\infty)}.
\end{align}
Inequalities \eqref{54}, \eqref{fi3} and \eqref{sep105} imply that there exists a constant $c$ such that
\begin{align}
    \label{sep106}
    c&\geq \sup_{u\in V^{s,A}_{d,1}(\rn)}\, \sup_{x\in\rn} \frac{|u(x)-u(0)|}{\omega(|x|) \, \big|\nabla ^{[s]} u \big|_{\{s\}, A, \rn}} \\ \nonumber & \gtrsim
        \sup_{f} \, \sup_{\rho>0}\frac{\int_0^{\omega_n\, \rho^n} f(r)\, r^{-1 +\frac sn} \; dr}{\omega (\rho)\, \|f\|_{L^A(0, \infty)}}=
\sup_{\rho>0}\, \sup_{f} \frac{\int_0^{\omega_n\, \rho^n} f(r)\, r^{-1 +\frac sn} \; dr}{\omega (\rho)\, \|f\|_{L^A(0, \infty)}}.
\end{align}
Starting from equation \eqref{sep106} instead of \eqref{july18} and arguing as in the proof of Theorem \ref{thm:s<1} yields
\begin{align}\label{sep107}
    \vartheta_s(\rho) \lesssim \omega(\rho) \qquad \text{for  $\rho > 0$.}
\end{align}
In particular, the latter inequality  and equation \eqref{july19} (which holds for every $s\in (1,n)$) imply  that
$$
    \int ^\infty  \frac{\widetilde{A}(t)}{t^{1+\frac n{n-s}}}\; dt<\infty.
$$
Hence, condition \eqref{convinf}
must be fulfilled.
\\ Next, given any function $f$ as above, consider the function $u \colon  \rn \to \R$ defined as in \eqref{sep108}. Notice that $u(0)=0$. Hence,
\begin{align}
    \label{sep100}
    |u(x_1, 0, \cdots , 0)- u(0, 0, \cdots, 0)|
& =\frac {1}{[s]!} \, |x_1| \,\int_{\omega_n |x_1|^n}^\infty f(r)\,  r^{-[s]-1+ \frac{s-1}{n}} \, (r - \omega_n |x_1|^n)^{[s]} \; dr
\\ \nonumber & \geq
\frac {1}{[s]!} 2^{1-[s]}\, |x_1| \,\int_{2\omega_n |x_1|^n}^\infty f(r)\,  r^{-1+ \frac{s-1}{n}}  \; dr
\qquad \hbox{for}\;\; x_1\in \R,
\end{align}
where the equality follows from formula \eqref{sep108}, via Fubini's theorem.
The following chain holds:
\begin{align}\label{131}
c&\geq \sup_{ u\in V^{s, A}_{d,1}(\rn) }\frac{|u(x_1, 0, \cdots , 0)- u(0, 0, \cdots, 0)|}{\omega (|x_1|)|\nabla ^{[s]} u|_{\{s\}, A, \rn}}
\\ \nonumber &
\gtrsim\sup_{ f\downarrow, \,{  \texttt{supp} f \subset[0,L]}}\, \sup_{r >0}\frac {r^{1/n}}{\omega( r^{1/n})} \, \frac{ \int_{2 \omega_nr}^\infty f(\rho) \, \rho ^{-1+\frac{s-1}n}\,d\rho}{\|f\|_{L^A(0,\infty)}}
\\ \nonumber
&  =\sup_{r >0} \frac {r^{1/n}}{\omega( r^{1/n})} \, \sup_{ f\downarrow, {  \texttt{supp} f \subset[0,L]}}\, \frac{ \int_{2 \omega_nr}^\infty f(\rho) \, \rho ^{-1+\frac{s-1}n}\,d\rho}{\|f\|_{L^A(0,\infty)}}
\\ \nonumber &
   \approx \sup_{r >0} \frac {r^{1/n}}{\omega( r^{1/n})} \, \sup_{ f\geq 0, {  \texttt{supp} f \subset[0,L]}}\, \frac{ \int_{2 \omega_nr}^\infty f(\rho) \, \rho ^{-1+\frac{s-1}n}\,d\rho}{\|f\|_{L^A(0,\infty)}}
\\
\nonumber & \approx  \sup_{r >0} \frac {r^{1/n}}{\omega( r^{1/n})} \, \| \rho ^{-1+\frac{s-1}n}{ \, \chi_{[0,L]}(\rho)}\|_{L^{\widetilde A}({2 \omega_nr}, \infty)},
\end{align}
where the first inequality holds by
\eqref{54}, the second by \eqref{sep100}, the third by \eqref{sep109}, the first equivalence by Lemma \ref{lemma2} and the last equivalence by \eqref{revholder}.
\\
 Letting $L \to\infty$ in inequality \eqref{131} and making use of the Fatou property of Luxemburg norms tell us that
\begin{align}\label{131'}
c\geq \sup_{r >0}\frac 1{\omega(r)} \, r \| \rho ^{-1+\frac{s-1}n}\|_{L^{\widetilde A}({2 \omega_nr^n}, \infty)}\approx \sup_{r >0} \frac {\varrho_s ((2 \omega_n)^{\frac 1n} r)}{\omega( r)}
\approx \sup_{r >0} \frac {\varrho_s (r)}{\omega( r)}
\end{align}
for some constant $c$, where the first equivalence holds by equation \eqref{july34}, and the second by the definition of the function $\varrho_s$ and property \eqref{min} applied to $\FF$. Hence,
\begin{align}
    \label{sep113}
    \varrho_s(r) \lesssim \omega(r) \qquad \text{for  $r> 0$.}
\end{align}
In particular,  inequality \eqref{sep113} and equation \eqref{july34} tell us that
\begin{align}\label{sep120}
 \Big \| \rho^{-1+ \frac {s-1}n}\, \chi_{(r^n, \infty)}(\rho) \Big\|_{L^{\widetilde{A}}(0, \infty) }<\infty \qquad \text{for $r >0$.}
\end{align}
Since the finiteness of the norm in \eqref{sep120} is in turn equivalent to condition \eqref{conv0'}, the necessity of the latter follows.
\\
Coupling inequality \eqref{sep107} with \eqref{sep113} yields
\begin{align}\label{sep125}
  \sigma_s(r) \lesssim \omega(r) \qquad \text{for  $r> 0$,}
\end{align}
whence the optimality of the space $\mathcal{C}^{\sigma_s(\cdot)} (\rn)$ follows.
\\ Assume now that
 condition \eqref{24} is in force. Inequality \eqref{54} entails that
 \begin{align}
     \label{sep122}
     r \lesssim \omega (r) \qquad \text{if $0<r\leq1$.}
 \end{align}
 Moreover, the same argument as in the proof under condition \eqref{july30} implies that inequalities \eqref{sep107} and \eqref{sep125} continue to hold.
 Coupling these inequalities with \eqref{sep122} shows that
 \begin{align}
     \label{sep126}
\sigma_s (r) \lesssim \omega (r) \qquad \text{for $r>0$.}
 \end{align}
 This establishes the optimality of the space $\mathcal{C}^{\sigma_s(\cdot)} (\rn)
$. \\ Finally,
 as observed above, inequalities \eqref{sep107} and \eqref{sep125} imply that properties \eqref{convinf}
and \eqref{conv0'} have to be satisfied.
\end{proof}

\subsection{Proof of Theorem \ref{thm:n<s<n+1}: case $s\in (n,n+1)$}

The outline of the proof of Theorem \ref{thm:n<s<n+1} is similar to that of Theorem \ref{thm:1<s<n}. We shall thus focus on the novel aspects and refer to that proof when analogous arguments have to be employed.

\begin{proof}[Proof of Theorem \ref{thm:n<s<n+1}] Unless otherwise stated, the constants involved in this proof depend on $n$, $s$ and $A$.
\\ The fact that $\sigma_s$ is a modulus of continuity is a consequence of properties (ii) and (iii) of Proposition \eqref{prop:F}.
\\ Assume first that
 condition \eqref{july30} holds. The same argument as in the proof of Theorem \ref{thm:1<s<n} yields
inequality \eqref{july50}, with $\nu_s$ given by equation \eqref{33}. Thereby, embedding \eqref{23n<s<n+1} will follow if we show that
\begin{equation}\label{55}
\nu_s(r) \approx \varrho_s (r)\qquad   \hbox{for $r > 0$.}
\end{equation}
Thanks to equivalence \eqref{july57}, the proof of   equation \eqref{55} is reduced to proving
that
\begin{equation}\label{56}
\Big \| \rho^{-1+ \frac sn}\, \chi_{(0, r^n)}(\rho) \Big\|_{L^{\widetilde{A}}(0, \infty)}\lesssim\, \varrho_s (r) \qquad \hbox{for}\;\; r>0.
\end{equation}
We have that
\begin{align}\label{57}
\Big \| \rho^{-1+ \frac sn}\, \chi_{(0, t)}(\rho) \Big\|_{L^{\widetilde{A}}(0, \infty)}= \inf\bigg\{ \lambda >0: \int_0^t \widetilde{A}\bigg(\frac{\tau^{\frac {s-n}n}}{\lambda}\bigg) \; d\tau \leq 1\bigg\} \qquad \hbox{for $t> 0$}.
\end{align}
Since $\frac{s-n}n>0$,
\begin{align}\label{58}
&\int_0^t \widetilde{A}\bigg(\frac{\tau^{\frac {s-n}n}}{\lambda}\bigg) \; d\tau = \frac n{s-n} \, \lambda^{\frac n{s-n}} \int_0^{\frac{t^{\frac {s-n}{n}}}{\lambda}} \widetilde{A} (\tau) \, \tau^{\frac{n}{s-n} -1}\; d\tau
\\ &
= \frac n{s-n} \,\bigg(\frac{t^{\frac {s-n}{n}}}{\lambda}\bigg)^{\frac n{n-s}} \, t \int_0^{\frac{t^{\frac {s-n}{n}}}{\lambda}}
\widetilde {A}(\tau) \, \tau^{\frac{n}{s-n} -1}\; d\tau =
\frac n{s-n} \, t \, I\bigg(\frac{t^{\frac {s-n}{n}}}{\lambda}\bigg) \qquad \text{for $t > 0$,} \nonumber
\end{align}
where $I\colon  (0, \infty) \to [0, \infty)$ is the Young function defined as
\begin{equation}\label{59}
    I(t)= t^{\frac n{n-s}} \int_0^t \widetilde {A}(\tau) \, \tau^{\frac{n}{s-n} -1}\; d\tau \qquad \hbox{for}\;\; t >0.
\end{equation}
Equations \eqref{57}, \eqref{58} and \eqref{min} tell us that
\begin{equation}\label{60}
\Big \| \rho^{-1+ \frac sn}\, \chi_{(0, r^n)}(\rho) \Big\|_{L^{\widetilde{A}}(0, \infty)}=\frac 1{ r^{n-s} \, I^{-1} \left(\frac{s-n}{n}\,  r^{-n}\right)}\approx \frac 1{ r^{n-s} \, I^{-1} \left( r^{-n}\right)} \qquad \text{for $r > 0$,}
\end{equation}
with equivalence constants depending on $n$ and $s$.
By property \eqref{AAtilde},
$$\frac {1}{r^{n-s}\, \FF^{-1}(r^{-n})} \leq \varrho_s (r)\leq \frac {2}{r^{n-s}\, \FF^{-1}(r^{-n})} \qquad \text{for $r>0$.}$$
Inequality \eqref{56} is therefore equivalent to
\begin{equation}\label{61}
    \frac {1}{r^{n-s}\, I^{-1}(r^{-n})} \leq    \frac {c}{r^{n-s}\, \FF^{-1}(r^{-n})}
    \qquad \hbox{for} \;\; r>0,
\end{equation}
for some constant $c=c(n,s)$.
 Inequality \eqref{61} is in turn equivalent to
\begin{equation}\label{62}
\FF(t) \geq I(t/c)
\qquad \hbox{for} \;\; t>0.
\end{equation}
Since the function $\widetilde{A}(r)\, r^{\frac n{s-n}-1}$ is increasing,
\begin{equation}\label{63}
    I(t) = t^{\frac n{n-s}}
\int_0^t \widetilde{A}(\tau)\, \tau^{\frac n{s-n}-1}\; d\tau \leq
 t^{\frac n{n-s}}
\widetilde{A}(t)\, t^{\frac n{s-n}-1} \int_0^t \; d\tau = \widetilde{A}(t) \qquad \text{for $t>0$.}
\end{equation}
Moreover,
\begin{align}\label{64}
    \FF(t) &= t^{\frac n{n-(s-1)}}
\int_0^t \frac{\widetilde{A}(\tau)}{\tau^{1+ \frac{n}{n-(s-1)}}}\; d\tau \geq
 t^{\frac n{n-(s-1)}} \int_{\frac t2}^t \frac{\widetilde{A}(\tau)}{\tau^{1+ \frac{n}{n-(s-1)}}}\; d\tau
 \\ &
  \geq
 t^{\frac n{n-(s-1)}}  \widetilde{A} (t/2) \int_{\frac t2}^t \tau^{-1-\frac{n}{n-(s-1)}} \;d\tau
 = c \, \widetilde{A} (t2) \geq
 \widetilde{A} (c' \, t)\nonumber \qquad \text{for $t>0$,}
\end{align}
and for some constant $c=c(n,s)$, where $c'=\frac{\min\{1, c\}}2$.
Inequality \eqref{62} follows from \eqref{63} and \eqref{64}.
\\
Next, suppose that condition\eqref{24} is in force. The same argument as in the proof of Theorem \ref{thm:1<s<n} tells us that embedding \eqref{july58} holds with $\mu_s$ satisfying \eqref{39}, where now $\nu_s$ fulfills  \eqref{55}. Embedding \eqref{23n<s<n+1} hence follows via properties \eqref{sep12} and \eqref{sep12'}.
\\ The optimality of the modulus of continuity in embedding \eqref{23n<s<n+1} and the necessity of condition \eqref{conv0'} can be established via the same argument as in the proof of Theorem \ref{thm:1<s<n}. In the present situation, one has only to make use of trial functions of the form \eqref{sep108}.
\end{proof}

\bigskip

\bigskip{}{}

 \par\noindent {\bf Data availability statement.} Data sharing not applicable to this article as no datasets were generated or analysed during the current study.

\section*{Compliance with Ethical Standards}\label{conflicts}

\smallskip
\par\noindent
{\bf Funding}. This research was partly funded by:
\\ (i) GNAMPA   of the Italian INdAM - National Institute of High Mathematics (grant number not available)  (A. Alberico, A. Cianchi);
\\ (ii) Research Project   of the Italian Ministry of Education, University and
Research (MIUR) Prin 2017 ``Direct and inverse problems for partial differential equations: theoretical aspects and applications'',
grant number 201758MTR2 (A. Cianchi);
\\ (iii) Research Project   of the Italian Ministry of Education, University and
Research (MIUR) Prin 2022 ``Partial differential equations and related geometric-functional inequalities'',
grant number 20229M52AS, cofunded by PNRR (A. Cianchi);
\\  (iv) Grant no.\ 23-04720S of the Czech Science
Foundation  (L. Pick, L. Slav\'ikov\'a);
\\ (v) Primus research programme PRIMUS/21/SCI/002 of Charles University (L. Slav\'ikov\'a);
\\ (vi) Charles University Research Centre program No.\ UNCE/24/SCI/005 (L. Slav\'ikov\'a).

\bigskip
\par\noindent
{\bf Conflict of Interest}. The authors declare that they have no conflict of interest.

\

\end{document}